\documentclass[11pt]{article}

\usepackage[a4paper,margin=2.5cm]{geometry}

\usepackage{lipsum}
\usepackage{amsfonts}
\usepackage{graphicx}
\usepackage{epstopdf}
\usepackage{algorithm}
\usepackage{algorithmic}
\usepackage{amsmath,amssymb,amsfonts,amsthm}
\usepackage{amsopn}
\usepackage{bbm}
\usepackage{mathrsfs}
\usepackage{hyperref}
\usepackage[nameinlink,noabbrev]{cleveref}
\usepackage{xcolor}
\usepackage[normalem]{ulem}
\usepackage{cancel}
\usepackage{comment}

\usepackage{todonotes}

\theoremstyle{plain}
\newtheorem{theorem}{Theorem}[section]
\newtheorem{lemma}[theorem]{Lemma}
\newtheorem{proposition}[theorem]{Proposition}
\newtheorem{corollary}[theorem]{Corollary}
\theoremstyle{definition}
\newtheorem{definition}[theorem]{Definition}
\newtheorem{assumptions}[theorem]{Assumptions}
\theoremstyle{remark}
\newtheorem{remark}[theorem]{Remark}

\numberwithin{equation}{section}

\newenvironment{keywords}{\par\smallskip\noindent\textbf{Keywords. }}{\par\smallskip}
\newenvironment{MSCcodes}{\par\smallskip\noindent\textbf{MSC Classification. }}{\par\smallskip}
\newcommand{\email}[1]{\texttt{#1}}

\title{Stochastic Optimal Control of Hawkes Jump-Diffusion Systems}
\author{Daria Sakhanda\thanks{ETH Z\"urich, Department of Mathematics, Switzerland (\email{daria.sakhanda@math.ethz.ch}).}
	\and Joshué Helí Ricalde-Guerrero\thanks{ETH Z\"urich, Department of Mathematics, Switzerland (\email{joshue.ricalde@math.ethz.ch}).}}

\def\n{\noindent}

\makeatletter

\providecolor{lyxadded}{rgb}{0,0,1}
\providecolor{lyxdeleted}{rgb}{1,0,0}

\begin{document}
	\date{}
	\maketitle
	
\begin{abstract}
    This paper is devoted to developing a framework for stochastic growth models with environmental risk, in which rare but catastrophic shocks interact with capital accumulation and pollution. Building on the Poisson point process formulation studied in \cite{sakhanda2025infinitehorizon}, we extend the model to disasters driven by a marked Hawkes process, allowing past disasters to temporarily increase the likelihood of subsequent shocks. Our work focuses on a subcritical Markovian Hawkes specification, in which the state space is augmented by the self-excitation component of disaster risk. We establish the well-posedness and nonexplosion of the resulting controlled Hawkes dynamics. Using the Hamilton-Jacobi-Bellman characterization of the corresponding Poisson control problem, we obtain quantitative estimates for the Hawkes excitation process and prove, under a small-excitation scaling, that the Hawkes value function converges to its Poisson counterpart as the magnitude of self-excitation vanishes. This provides a rigorous Poisson approximation of the stochastic control problem and quantifies the effect of self-excitation on optimal growth under environmental disaster risk.
\end{abstract}



\section{Introduction}

\n 

The study of economic growth under environmental risk raises fundamental  mathematical questions in stochastic control for systems combining continuous  dynamics and discontinuous jumps. Classical growth-environment models typically describe pollution and damages as deterministic or smoothly evolving, with dynamics  governed by stochastic differential equations without discontinuities. 
Empirical evidence, however, shows that risk also takes the form of rare but catastrophic events -- such as abrupt climate disasters, ecosystem collapses, or large-scale technological failures -- which arrive unpredictably and are naturally modeled as jumps with large economic impacts. While Poisson-based models provide a natural framework for rare catastrophic events, they do not account for temporal clustering: the occurrence of one disaster may increase the likelihood of subsequent events, generating periods of persistently elevated risk.

Motivated by economic questions of growth, environmental change, and climate risk, this paper extends our previous Poisson framework \cite{sakhanda2025infinitehorizon} by introducing self-exciting disaster arrivals through a marked Hawkes process. The disaster intensity combines a pollution-dependent baseline component with an excitation component generated by past disasters. Each event increases the excitation level according to its mark, while excitation gradually decays between successive events. The resulting framework therefore allows disaster risk to depend simultaneously on the current environmental state and on the history of previous disasters.

From the mathematical perspective, the Hawkes specification introduces an additional state variable and an endogenous jump intensity, leading to a stochastic control problem on an augmented state space. We establish the global strong well-posedness and non-explosion of the controlled marked Hawkes system, as well as a rigorous connection between the Hawkes and Poisson formulations. We consider a family of Hawkes models in which the excitation generated by individual disasters becomes progressively smaller. Under suitable moment and stability conditions, we derive uniform discounted estimates for the excitation process and show that the Hawkes value function converges to the corresponding Poisson value function in the small-excitation limit. Moreover, we obtain an explicit first-order bound for the approximation error.

Thus, the main contributions of the paper are threefold:
\begin{enumerate}
	\item We introduce a controlled marked Hawkes specification for pollution-dependent and self-exciting catastrophic risk in a growth-environment model.
	\item We establish the mathematical foundations of the resulting stochastic control problem, including global strong well-posedness and non-explosion.
	\item We establish a quantitative connection between the Hawkes and Poisson formulations by proving convergence of their value functions in the small-excitation regime, together with an explicit rate of approximation.
\end{enumerate}

The rest of the paper is organized as follows. In Section~2, we introduce the modelling framework and describe how the state-dependent Poisson disaster mechanism studied in \cite{sakhanda2025infinitehorizon} is extended to a marked Hawkes specification with self-excitation. Section~3 collects the basic concepts and notation for marked Hawkes processes and introduces the subcriticality condition used throughout the analysis. In Section~4, we formulate the controlled Hawkes model, specify the admissible controls and standing assumptions, and state our main result. Section~5 is devoted to the Hawkes-Poisson approximation. We first establish the well-posedness and nonexplosion of the controlled Hawkes dynamics and then study the asymptotic behaviour of the excitation process and its implications for the disaster mechanism. Finally, we introduce a small-excitation scaling, parameterized by $\eta>0$, where $\eta$ controls the magnitude of Hawkes self-excitation, and prove that the Hawkes value function converges to its Poisson counterpart as $\eta\downarrow0$, with an explicit $O(\eta)$ approximation bound.

\section{Modelling framework}

\n 

We consider a stylized representation of the global economy, which produces a single composite good under constant returns to scale. Production relies on the aggregate capital stock at time $t$, denoted by $K_t$ which includes physical capital, human capital, and intangible assets. The production process $(K_t)_{t \geq 0}$ generates pollution: at each instant $t$, greenhouse gas (GHG) emissions $E_t$ are released into the atmosphere. These emissions accumulate in the atmospheric pollution stock, $P_t$, which increases with the flow of emissions and decreases through a natural absorption rate $\alpha \in [0,1)$, assumed to be small or negligible. While $P_t$ is referred to as the pollution stock, it can more generally be interpreted as the \textit{inverse of environmental quality}. Similarly, the emissions variable $E_t$, can be viewed more broadly as any environmentally damaging by-product of economic activity. 

In this paper, we extend the disaster arrival mechanism from state-dependent Poisson processes, as studied in \cite{sakhanda2025infinitehorizon}, to a Hawkes process with self-excitation. This specification captures the empirically observed clustering of extreme events, whereby the occurrence of a disaster increases the likelihood of subsequent disasters in the short run, before this excess risk gradually decays over time (e.g., aftershocks, persistent weather regimes, or cascading failures). 

Following the notation from \cite{sakhanda2025infinitehorizon}, the controlled dynamics of the model are
\begin{align}
	dK_t &= b^{\mathrm{cap}}(K_t,P_t,C_t,\theta_t)\,dt - \int_{\mathcal Z} \bigl(1-\omega(K_{t-},P_{t-},\zeta)\bigr)K_{t-}\, N(dt,d\zeta), 
    \label{eq:marked_hawkes_K} 
    \\
	dP_t &= b^{\mathrm{pol}}(K_t,P_t,C_t,\theta_t)\,dt + \sigma P_t\,dW_t, 
    \label{eq:marked_hawkes_P} 
    \\
	dY_t &= -bY_t\,dt + \int_{\mathcal Z}a(\zeta)\,N(dt,d\zeta). \label{eq:marked_hawkes_Y}
\end{align}
Here, $Y=(Y_t)_{t\geq0}$ denotes the self-exciting component of the disaster intensity, while $N$ is the marked integer-valued random measure associated with disaster arrivals. The process $W=(W_t)_{t\geq0}$ is a Brownian motion, independent of the underlying Poisson random measure used in the construction of $N$, and $\sigma>0$ determines the magnitude of the diffusion component in the pollution dynamics. Moreover, $b>0$ is the decay rate of self-excitation, whereas $a(\zeta)\geq0$ measures the increase in excitation generated by a disaster with mark $\zeta$. The processes $Y$ and $N$ will be specified more precisely below.

In comparison, the marked PRM-driven model from Section~2.4 in \cite{sakhanda2025infinitehorizon}, follows the same dynamics on the capital and pollution \eqref{eq:marked_hawkes_K} and \eqref{eq:marked_hawkes_P}, respectively, with no additional self-exciting component \eqref{eq:marked_hawkes_Y}. 
The Hawkes specification can therefore be viewed as an extension of the Poisson model in which past disasters affect the likelihood of subsequent shocks through the excitation process $Y$. Under the subcriticality condition, we establish the well-posedness and nonexplosion of the Hawkes-driven model. We then introduce the small-excitation scaling
$$
a_\eta(\zeta):=\eta a(\zeta), \qquad \eta>0,
$$
and show that, as $\eta\downarrow0$, the corresponding Hawkes control problem converges to its Poisson counterpart.

The predictable disaster intensity is
\begin{equation} \label{eq:lambda_Hawkes}
    \lambda_t^H = \mu(P_{t-}^H)+Y_{t-}.
\end{equation}
Thus, $\mu(P_{t-}^H)$ represents the pollution-dependent baseline intensity, while $Y_{t-}$ captures the additional disaster risk generated by past events. At each disaster time with mark $\zeta$, the excitation process increases by $a(\zeta)$ and subsequently decays at rate $b$. Consequently, past disasters affect future disaster arrivals through $Y$, generating the characteristic self-exciting behaviour of the Hawkes process.

In comparison, the marked Poisson model of \cite{sakhanda2025infinitehorizon} has disaster intensity
\begin{equation} \label{eq:lambda_Poisson}
\lambda_t^P=\mu(P_{t-}^P),
\end{equation}
and therefore contains no self-exciting component. The Hawkes specification can thus be viewed as an extension of the Poisson benchmark in which past disasters generate an additional endogenous component of disaster risk. A central objective of the present paper is to quantify the effect of this component and to establish conditions under which the corresponding Hawkes control problem can be approximated by its Poisson counterpart.

\section{Preliminaries}

\subsection{Basic concepts and notation}

\n 

Throughout the paper, we work on a filtered probability space
$(\Omega,\mathcal F,\mathbb F,\mathbb P)$, where
$\mathbb F=(\mathcal F_t)_{t\geq0}$ satisfies the usual conditions.
For a c\`adl\`ag process $X=(X_t)_{t\geq0}$, we denote by
$$
X_{t-}:=\lim_{s\uparrow t}X_s
$$
its left limit at time $t>0$, and by
$$
\Delta X_t:=X_t-X_{t-}
$$
its jump at time $t$.

Let $(\mathcal Z,\mathfrak Z)$ be a measurable mark space, endowed with a probability measure $\nu$; see \cite{daley_introduction_2008}. The variable $\zeta\in\mathcal Z$ will represent the mark associated with a disaster, describing its magnitude or type. Integration with respect to the mark distribution will be written as
$$
\int_{\mathcal Z} f(\zeta)\,\nu(d\zeta)
$$
whenever the integral is well defined.

For a measurable space $(E,\mathcal E)$, we denote by $\mathcal B(E)$ its Borel $\sigma$-algebra whenever $E$ is endowed with a topology. The indicator function of a measurable set $A$ is denoted by $\mathbf 1_A$.

The state spaces of the Poisson and Hawkes models are denoted, respectively, by
$$
\mathcal S_P:=\mathbb R_{>0}\times\mathbb R_{>0},
\qquad
\mathcal S_H:=\mathcal S_P\times\mathbb R_{\geq0}.
$$
Accordingly, the Poisson state is represented by $(K,P)$, whereas the Hawkes state is augmented by the excitation component $Y$ and is represented by $(K,P,Y)$.

The following subsection recalls the point-process notation and the basic properties of marked Hawkes processes required throughout the paper.

\subsection{Preliminaries on Hawkes processes}

\n 

Let $\{(\tau_n,\Delta_n)\}_{n\geq1}$ be a sequence of random variables
taking values in $\mathbb R_{\geq0}\times\mathcal Z$, where
$$
0<\tau_1<\tau_2<\cdots
$$
denote the event times. We refer to $\tau_n$ as the \emph{arrival time of the $n$-th event} (disaster) and to $\Delta_n$ as
the \emph{mark of the $n$-th event} (magnitude or type of disaster).
Following \cite{daley_introduction_2008}, let $\mathcal{ N }_{ \mathbb{R} \times \mathcal{Z} }^\#$ denote space of integer-valued measures over $\mathbb{R} \times \mathcal{Z}$ and define $N$ as the measurable mapping
\begin{align*}
    N : (\Omega, \mathcal{F}) 
        &\longrightarrow 
        \big( 
            \mathcal{ N }_{ \mathbb{R} \times \mathcal{Z} }^\#, \mathcal{B}(\mathcal{ N }_{ \mathbb{R} \times \mathcal{Z} }^\#)
        \big)
    \\
    \omega
        &\longmapsto
        \sum_{n \geq 1} \delta_{(\tau_n(\omega), \Delta_n(\omega))} (dt, d\zeta).
\end{align*}
We refer to both the sequence of events $\{(\tau_n, \Delta_n)\}$ and the (random) counting measure $N$ as \emph{marked point process} interchangeably.

Integrals with respect to the point process are defined (realization-by-realization) as 
\begin{align*}
    \int_{{\mathbb{R}} \times \mathcal{Z}} f(t,\zeta) N(\omega;dt \times d\zeta) = \sum_{n \geq 1} f(\tau_n(\omega), \Delta_n(\omega))
\end{align*}
for any measurable function $f$ with $\operatorname{dom}f \subseteq \mathbb{R}_{\geq 0} \times \mathcal{Z}$. In particular, the \emph{ground counting process} \cite{daley_introduction_2008} associated with $N$ is defined as
\begin{align}
    \label{Eq:Ground-process}
    &N_t := \int_{{\mathbb{R}} \times \mathcal{Z}} 1_{(0,t] \times \mathcal{Z}}(t,\zeta) \, N(dt \times d\zeta) = N((0,t] \times \mathcal{Z}),
    &
    & \forall t \geq 0,
\end{align}
so that $N_t$ records the total number of disasters occurring up to time $t$.
Let $\mathbb F=(\mathcal F_t)_{t\geq0}$ be filtration satisfying the usual conditions, such that $\mathbb{F}$ contains the filtration generated by $\{N((0,s] \times Z);\, s \leq t, Z \in \mathfrak{Z}\}_{t \geq 0}$. From hereonout, we say the point process $\{(\tau_n, \Delta_n)\}$ -- equivalently, $N$ -- is a \emph{marked Hawkes process} if there exists an $\mathbb{F}$-predictable \textit{stochastic intensity kernel} $\Lambda$ satisfying the following self-exciting jump SDE:
\begin{equation}
    \label{eq:Marked_Hawkes_Y-1}
    \begin{aligned}
        \Lambda(\omega; t, d\zeta) 
            &=
            \lambda(\omega; t,\zeta) \,\nu(d \zeta),
        \\ \vphantom{\int}
        \lambda(\omega; t,\zeta)
            &=
            \mu(\omega; t) + 
            \alpha(Y_t(\omega),\zeta)
        \\ 
        Y_t(\omega) 
            &= 
            -b \int_{(0,t]} Y_s(\omega) ds + 
            \int_{(-\infty,t] \times \mathcal{Z}} a(\omega; t-s,\zeta) N(\omega; \operatorname{d} s \times \operatorname{d} \zeta ),
    \end{aligned}
\end{equation}
where $\nu$ is some given probability measure over $(\mathcal{Z}, \mathfrak{Z})$, and $\mu$ is some given $\mathbb{F}$-predictable process with values in $\mathbb{R}_{\geq 0}$. 
Thus, each disaster increases the future arrival intensity through the excitation process $Y$, while this additional excitation decays exponentially between successive events. The strength of this self-exciting mechanism must be sufficiently small relative to its rate of decay in order to ensure stability of the Hawkes process. This leads to the subcriticality condition introduced next.

\subsection{Subcriticality condition}

\n 

A fundamental stability condition for Hawkes processes is the
\emph{subcriticality condition}. In the cluster representation of a linear Hawkes process, this condition requires the expected number of offspring generated by a single event to be strictly smaller than one; see \cite{hawkes_oakes_1974} and \cite{bremaud_stability_1996}.

\begin{definition}[Subcritical marked Hawkes process]
	\label{def:subcritical_hawkes}
	Let
	$$
	h:\mathbb R_{>0}\times\mathcal Z\longrightarrow\mathbb R_{\geq0}
	$$
	be a measurable excitation kernel satisfying
	$$
	\int_0^\infty\int_{\mathcal Z}
	h(s,\zeta)\,\nu(d\zeta)\,ds<\infty.
	$$
	The quantity
	\begin{equation}
		\label{eq:hawkes_branching_ratio}
		m_h
		:=
		\int_0^\infty\int_{\mathcal Z}
		h(s,\zeta)\,\nu(d\zeta)\,ds
	\end{equation}
	is called the \emph{branching ratio} and represents the expected number
	of direct offspring generated by a single event. The marked Hawkes
	process is said to be \emph{subcritical} if
	\begin{equation}
		\label{eq:general_hawkes_subcriticality}
		m_h<1.
	\end{equation}
	See \cite{hawkes_oakes_1974} for the cluster representation of linear
	Hawkes processes.
\end{definition}

\section{Controlled Hawkes model and main results}

\n 

We start by defining the set of admissible controls and specifying the standing assumptions. 
The \emph{set of admissible actions} corresponds to a given 
subset
\begin{align*}
    \mathfrak{a} \subseteq \mathbb{R}_{\geq 0} \times [0,1].
\end{align*}

\begin{definition}[Admissible controls]
For every $(k,p,y) \in \mathcal{S}_H$, the set of admissible controls is defined as
\begin{align*}
    \mathcal{A}{(k,p,y)}
    :=
    \Big\{
        \alpha:\Omega \times \mathbb{R}_{\geq 0} \longrightarrow \mathfrak{a} 
        \,\Big|\,
        &(\alpha_t)_{t\geq 0} \text{ is }\mathbb{F}-\text{predictable,}
        \\ \vphantom{\int}
        &(K_0,P_0, Y_0) = (k,p,y) \in \mathcal{F}_0,
        \\
        &\text{and }\mathbb{E} \Big[ \int_{0}^\infty \|\alpha_t\|^2 d t \Big] < \infty
    \Big\}.
\end{align*}
\end{definition}

In the context of the problem, the marginals of an admissible control correspond to the consumption and abatement processes $(C_t)_{t \geq 0}$ and $(\theta_t)_{t \geq 0}$, respectively. 

\begin{assumptions}
\label{Assumptions}
\hfill
\noindent
\begin{enumerate}
    \item 
        \textbf{On the stochastic basis:}
        The filtered probability space $(\Omega, \mathcal{F}, \mathbb{F}, \mathbb{P})$ admits an $\mathbb{F}$-Brownian motion $W$, and an independent $\mathbb{F}$-Poisson random measure $M$ on $\mathbb{R}_{\geq 0} \times \mathcal{Z} \times \mathbb{R}_{\geq 0}$ with intensity measure $dt \otimes \nu(d \zeta) \otimes dr$ for some $\sigma$-finite measure $\nu$ with finite second moment. In addition, $\mathbb{F} = \{ \mathcal{F}_t \}_{t \geq 0}$ is (the complete, right-continuous augmentation of) the filtration
        \begin{align}
    \mathcal{F}_t
    :=
    \sigma\Big\{
        M\bigl((0,s]\times C\times D\bigr)
        ~\big|~
        0\leq s\leq t,\;
        C\in\mathfrak{Z},\;
        D\in\mathcal{B}(\mathbb{R}_{\geq0})
    \Big\}
    \vee
    \mathcal{F}_t^W,
    \qquad t\geq0.
\end{align}

    \item
        \textbf{On the coefficients:}
        \begin{enumerate}
        
        	\item
        		$\mu:\mathbb R_{>0}\to\mathbb R_{\geq0}$ is locally Lipschitz and there exists $L_\mu>0$ such that
        		\begin{equation}
                    \label{eq:mu_growth_assumpt}
                    \mu(p)\leq L_\mu(1+p), \qquad p>0.
        	    \end{equation}
        		
        	\item
        		$a:\mathcal Z\to\mathbb R_{\geq0}$ is measurable and
        		\begin{align}
                    \int_{\mathcal Z}a(\zeta)\,\nu(d\zeta)<\infty,
                    \qquad
                    \int_{\mathcal Z}a(\zeta)^2\,\nu(d\zeta)<\infty.
                    \label{eq:int_a_a^2_finite}
                \end{align}

        	\item
        		For every admissible control $(C,\theta)$, the functions $b^{\mathrm{cap}}$ and $b^{\mathrm{pol}}$ are locally Lipschitz in $(k,p)$, uniformly with respect to the control on bounded sets, and satisfy a linear-growth condition
        		\begin{align}
                    |b^{\mathrm{cap}}(k,p,c,\theta)| + |b^{\mathrm{pol}}(k,p,c,\theta)| \leq L(1+k+p+c)
                \end{align}
        		for some $L>0$.
        				
            \item
        		The survival map
        		\begin{align}
                    \omega:
            		\mathbb R_{>0}\times\mathbb R_{>0}\times\mathcal Z
            		\longrightarrow(0,1]
            	\end{align}
            	is measurable, locally Lipschitz in $(k,p)$ uniformly in $\zeta$, and satisfies
            	\begin{align}
                    0<\omega(k,p,\zeta)\leq1.
                    \label{eq:w_at_disaster}
            	\end{align}
        		
        	\item
                Starting from $K_0>0$ and $P_0>0$, the continuous dynamics do not attain the boundary of $\mathbb R_{>0}^2$ in finite time; that is, $ K_t>0$, $P_t>0$ for every finite $t$ for which the continuous dynamics are defined.
        \end{enumerate}

    \item 
        \textbf{Subcriticality:}
Let
$$
\bar a_1:=\int_{\mathcal Z}a(\zeta)\,\nu(d\zeta).
$$
The excitation parameters satisfy
\begin{equation}
	\label{eq:hawkes_subcriticality_assumption}
	\bar a_1<b.
\end{equation}
For the exponential excitation kernel considered below, condition~\eqref{eq:hawkes_subcriticality_assumption} is precisely the specialization of the general subcriticality condition introduced in Definition~\ref{def:subcritical_hawkes}.
\end{enumerate}
\end{assumptions}
With the previous considerations in mind, the augmented system
$X^H=(X_t^H)_{t\geq0}$ is defined as the $\mathcal S_H$-valued,
$\mathbb F$-adapted process
\begin{align*}
    X_t^H=(K_t^H,P_t^H,Y_t),
    \qquad
    \forall t\geq0,
\end{align*}
characterized by the controlled dynamics:
\begin{align}
	&d K^H_t =b^{\mathrm{cap}}(K^H_t,P^H_t,\alpha_t)\, dt -
	\int_\mathcal{Z} \bigl(1-\omega(K^H_{t-},P^H_{t-}, \zeta)\bigr)K^H_{t-} \, N^H( d t \times d \zeta ),
	\label{eq:Marked_Hawkes_K}
	\\ \vphantom{\int}
	&dP^H_t = b^{\mathrm{pol}}(K^H_t,P^H_t,\alpha_t)\,dt + \sigma P^H_t d W_t,
	\label{eq:Marked_Hawkes_P}
    \\ \label{Eq:Marked_Hawkes_Initial_Conditions} \vphantom{\int}
    &(K^H_0, P^H_0, Y_0) = (k, p, y) \in \mathcal{S}_H.
\end{align}
Here $N^H$ is a \emph{marked Hawkes process} with ground process $N = (N_t)$, as defined in \eqref{Eq:Ground-process} with
\begin{align}
    \label{eq:Marked_Hawkes_Y-2}
    &\mu(\omega; t) = \mu(P^H_t(\omega))
    &
    &\text{and}
    &
    &a(\omega; t-s,\zeta) = a(\zeta),
\end{align}
where $\mu:\mathbb R_{>0}\to\mathbb R_{\geq0}$ and
$a:\mathcal Z\to\mathbb R_{\geq0}$ are the deterministic measurable functions specified in Assumptions~\ref{Assumptions}.

We now state our main result:
\begin{theorem}
\label{thm:hawkes_poisson_value_approximation}
Under Assumptions \ref{Assumptions}, the following hold:
\begin{enumerate}
    \item 
\textbf{Well-posedness of the Hawkes-driven model:}
Let $(k,p,y)\in\mathcal S_H$ be given. For any admissible control
$\alpha\in\mathcal A(k,p,y)$, there exists a unique
$\mathbb F$-adapted strong solution
$$
X^H=(K^H,P^H,Y)
$$
to \eqref{eq:Marked_Hawkes_K}--\eqref{Eq:Marked_Hawkes_Initial_Conditions},
where
$$
dY_t = -bY_t\,dt + \int_{\mathcal Z}a(\zeta)\,N^H(dt,d\zeta), \qquad t>0.
$$
Here $N^H$ denotes the integer-valued random measure associated with the marked Hawkes process, defined by the predictable thinning
$$
N^H(dt,d\zeta) := \int_{\mathbb R_{\geq0}} \mathbf 1_{[0,\mu(P_{t-}^H)+Y_{t-}]}(r) \,M(dt,d\zeta,dr),
$$
where $M$ is the underlying Poisson random measure on
$\mathbb R_{\geq0}\times\mathcal Z\times\mathbb R_{\geq0}$.
The predictable compensator of $N^H$ is
$$
\Lambda^H(dt,d\zeta) = \bigl(\mu(P_{t-}^H)+Y_{t-}\bigr) \nu(d\zeta)\,dt.
$$
       
       \item \textbf{Well-posedness of the Poisson-driven model
(see Section~4 in \cite{sakhanda2025infinitehorizon}):}
Let $(k,p)\in\mathcal S_P$ be given. For any admissible control
$\alpha\in\mathcal A_P(k,p)$, there exists a unique
$\mathbb F$-adapted process
$$
X^P=(K_t^P,P_t^P)_{t\geq0}
$$
solving
\begin{align}
	dK_t^P &= b^{\mathrm{cap}}(K_t^P,P_t^P,\alpha_t)\,dt -
	\int_{\mathcal Z}\bigl(1-\omega(K_{t-}^P,P_{t-}^P,\zeta)\bigr)
	K_{t-}^P\,N^P(dt,d\zeta), \label{eq:Poisson_comparison_K}
	\\
	dP_t^P &= b^{\mathrm{pol}}(K_t^P,P_t^P,\alpha_t)\,dt + \sigma P_t^P\,dW_t,
	\label{eq:Poisson_comparison_P}
	\\
	(K_0^P,P_0^P) &= (k,p)\in\mathcal S_P.
	\label{eq:Poisson_comparison_Initial_conditions}
\end{align}
Here $N^P$ denotes the integer-valued random measure associated with the Poisson disaster process, defined by the predictable thinning
$$
N^P(dt,d\zeta) := \int_{\mathbb R_{\geq0}} \mathbf 1_{\{r\leq\mu(P_{t-}^P)\}} \,M(dt,d\zeta,dr),
$$
where $M$ is the underlying Poisson random measure on
$\mathbb R_{\geq0}\times\mathcal Z\times\mathbb R_{\geq0}$. The predictable compensator of $N^P$ is
$$
\Lambda^P(dt,d\zeta) = \mu(P_{t-}^P)\nu(d\zeta)\,dt.
$$

    \item \textbf{Approximation property:}
Let $V_H^\eta$ denote the value function of the Hawkes model with
scaled excitation
$$
a_\eta(\zeta):=\eta a(\zeta).
$$
Then, for all sufficiently small $\eta>0$,
\begin{equation}
	\left|
	V_H^\eta(k,p,y_\eta)-V_P(k,p) \right| \leq \frac{\psi^{-\varepsilon}}{1-\varepsilon}
	A_*\eta,
\end{equation}
where $A_*>0$ is independent of $\eta$. Consequently,
$$
V_H^\eta(k,p,y_\eta)\longrightarrow V_P(k,p) \qquad\text{as }\eta\downarrow0.
$$
\end{enumerate}
\end{theorem}

\section{Hawkes-Poisson approximation}

\n 

This section is devoted to the proof of Theorem~\ref{thm:hawkes_poisson_value_approximation}. The main objective is to quantify the effect of Hawkes self-excitation on the associated infinite-horizon control problem and to show that, when the excitation becomes sufficiently small, the Hawkes value function approaches its Poisson counterpart.

The argument is developed in several stages. We first establish the well-posedness and nonexplosion of the controlled Hawkes dynamics. We then study the asymptotic behaviour of the excitation process and show that, under the subcriticality condition, the contribution of self-excitation vanishes asymptotically. This provides a first probabilistic connection between the Hawkes and Poisson models.
Finally, we introduce a small-excitation scaling and derive estimates that allow us to transfer this approximation to the corresponding infinite-horizon value functions.

Throughout this section, we focus on the arguments that are specific to the Hawkes specification. The corresponding Poisson control problem and its well-posedness properties were studied in \cite{sakhanda2025infinitehorizon}.

\subsection{Well-posedness of the Hawkes model}

\n 

We begin by establishing that the controlled state dynamics are well defined. This is necessary before considering either the asymptotic behaviour of the Hawkes excitation or the associated infinite-horizon control problem. The main difficulty relative to the Poisson model is that the jump intensity is itself stochastic and depends on the endogenous excitation process $Y$.

The following proposition establishes existence and pathwise uniqueness of the controlled state process and, importantly, excludes the accumulation of disaster times on finite intervals. The proof relies on a Poisson embedding of the marked Hawkes process, followed by localization and uniform finite-horizon moment estimates.

\begin{proposition}[Existence, pathwise uniqueness and nonexplosion]
\label{prop:wellposed_marked_hawkes}
Let $T>0$ be fixed, and let $(C_t,\theta_t)_{t \geq 0}$ be an arbitrary admissible control. Then, on every finite horizon $[0,T]$, there exists a pathwise unique strong solution $(K_t,P_t,Y_t)_{0\leq t\leq T}$ to the SDE 
\begin{align}
    \label{eq:WP_K}
    dK_t &= b^{\mathrm{cap}}(K_t,P_t,C_t,\theta_t)\,dt - \int_{\mathcal Z} \bigl(1-\omega(K_{t-},P_{t-},\zeta)\bigr)K_{t-} \,N(dt,d\zeta),
	\\ \label{eq:WP_P} \vphantom{\int}
	dP_t &= b^{\mathrm{pol}}(K_t,P_t,C_t,\theta_t)\,dt + \sigma P_t\,dW_t, 
    \\ \label{eq:WP_Y}
	dY_t &= -bY_t\,dt + \int_{\mathcal Z} a(\zeta)\,N(dt,d\zeta),
\end{align}
with initial condition $(K_0,P_0,Y_0)=(k,p,y)\in\mathcal S_H$. Moreover, for every $T>0$,
\begin{align*}
    N((0,T]\times\mathcal Z)<\infty
	\qquad\text{a.s.},
\end{align*}
Hence, the marked point process is locally finite and nonexplosive, and the global solution is locally square-integrable in time.
\end{proposition}

\begin{proof} 
Let us consider the proof divided in four steps:
\begin{enumerate}

\item \textit{Poisson embedding and local construction.} We first construct the marked point process by means of a Poisson embedding. Given a nonnegative predictable process
$$ \lambda_t = \mu(P_{t-})+Y_{t-}, $$
we define
\begin{equation}
	N(dt,d\zeta) = \int_{\mathbb R_{\geq0}} \mathbf 1_{\{r\leq\lambda_t\}} \,M(dt,d\zeta,dr).
	\label{eq:poisson_embedding_wp}
\end{equation}
By the standard Poisson embedding, or thinning, construction for point processes with predictable stochastic intensity, the random measure $N$ defined by \eqref{eq:poisson_embedding_wp} has predictable compensator
\begin{equation}
\Lambda(dt,d\zeta) = \lambda_t\,\nu(d\zeta)\,dt = \bigl(\mu(P_{t-})+Y_{t-}\bigr)\nu(d\zeta)\,dt 
\label{eq:predictable_compensator}
\end{equation}
See, for instance, the standard theory of point processes and Poisson embedding in \cite{bremaud_point_1981}; the same construction is used in the classical theory of Hawkes processes in \cite{bremaud_stability_1996}. 
We next localize the coupled system. Since
$$
\mathcal S_H = (0,\infty)\times(0,\infty)\times[0,\infty),
$$
we define, for $n\in\mathbb N$,
\begin{equation}
D_n := \left[\frac1n,n\right] \times \left[\frac1n,n\right] \times [0,n].	
\label{eq:D_n}
\end{equation}
Then
$$ D_n\subset\mathcal S_H, \qquad D_n\subset D_{n+1}, \qquad \bigcup_{n\geq1}D_n=\mathcal S_H. $$

By the assumed local Lipschitz continuity of $b^{\mathrm{cap}}$, $b^{\mathrm{pol}}$, $\mu$, and $\omega$, their restrictions to $D_n$ satisfy Lipschitz estimates with constants depending on $n$. We extend these restrictions outside $D_n$ to globally Lipschitz coefficients, denoted by
$$
b_n^{\mathrm{cap}}, \qquad
b_n^{\mathrm{pol}}, \qquad
\mu_n, \qquad
\omega_n,
$$
which coincide with the original coefficients on $D_n$.
We also truncate the stochastic intensity by setting
\begin{equation}
	\lambda_t^{(n)} := \left( \mu_n(P_{t-}^{(n)}) + Y_{t-}^{(n)} \right)\wedge n.
	\label{eq:truncated_intensity}
\end{equation}
The corresponding marked counting measure is defined by
\begin{equation}
	N^{(n)}(dt,d\zeta) = \int_{\mathbb R_{\geq0}} \mathbf 1_{\{r\leq\lambda_t^{(n)}\}} \,M(dt,d\zeta,dr).
	\label{eq:truncated_counting_measure}
\end{equation}
In particular, $0\leq\lambda_t^{(n)}\leq n,$ so the jump intensity of the truncated system is bounded.

The truncated state process
$$ X_t^{(n)} = (K_t^{(n)},P_t^{(n)},Y_t^{(n)}) $$
therefore solves
\begin{align}
	dK_t^{(n)}
	&= b_n^{\mathrm{cap}} (K_t^{(n)},P_t^{(n)},C_t,\theta_t)\,dt - \int_{\mathcal Z} \bigl( 1-\omega_n(K_{t-}^{(n)},P_{t-}^{(n)},\zeta ) \bigr) K_{t-}^{(n)} \,N^{(n)}(dt,d\zeta),
	\label{eq:truncated_K}
	\\
	dP_t^{(n)}
	&= b_n^{\mathrm{pol}} (K_t^{(n)},P_t^{(n)},C_t,\theta_t)\,dt + \sigma P_t^{(n)}\,dW_t,
	\label{eq:truncated_P}
	\\
	dY_t^{(n)}
	&= -bY_t^{(n)}\,dt + \int_{\mathcal Z} a(\zeta)\,N^{(n)}(dt,d\zeta).
	\label{eq:truncated_Y}
\end{align}

It is useful to note that, although the map
$ x\longmapsto \mathbf 1_{\{r\leq\lambda^{(n)}(x)\}} $
is not pointwise Lipschitz for fixed $r$, it satisfies the standard integrated estimate
\begin{equation}
	\int_0^\infty \left| \mathbf 1_{\{r\leq\lambda^{(n)}(x)\}} - \mathbf 1_{\{r\leq\lambda^{(n)}(x')\}} \right| \,dr = \left| \lambda^{(n)}(x)-\lambda^{(n)}(x') \right|.
	\label{eq:indicator_lipschitz}
\end{equation}
Indeed, for any $u,v\geq0$,
$$ \int_0^\infty \left| \mathbf 1_{\{r\leq u\}} - \mathbf 1_{\{r\leq v\}} \right|\,dr = |u-v|. $$
Since $\mu_n$ is globally Lipschitz and
$ (u,y)\longmapsto(u+y)\wedge n $
is Lipschitz, the right-hand side of \eqref{eq:indicator_lipschitz} is bounded by $ A_n|x-x'|$
for some constant $A_n>0$. Thus the thinning coefficient satisfies the appropriate integrated Lipschitz estimate required in the Brownian-Poisson SDE theory.
Together with $ a\in L^2(\nu) $ and the global Lipschitz and growth bounds of the truncated coefficients, this places \eqref{eq:truncated_K}--\eqref{eq:truncated_Y}within the standard strong well-posedness theory for stochastic differential equations driven jointly by Brownian motion and Poisson random measures. Consequently, for every $n\in\mathbb N$, the truncated system admits a pathwise unique strong c\`adl\`ag solution $X^{(n)}$ on $[0,T]$; see, 
for example, \cite{applebaum_levy_2009} and the classical treatment of SDEs driven by Poisson point processes in \cite{ikeda_watanabe_1989}.

Define the exit time
\begin{equation}
	\tau_n := \inf\left\{ t\geq0: X_t^{(n)}\notin D_n \right\}.
	\label{eq:exit_time_Dn}
\end{equation}
Since the truncated coefficients coincide with the original coefficients on $D_n$, the process $X^{(n)}$ solves the original system on $[0,\tau_n)$. Moreover, the truncated systems can be chosen consistently. In particular, if $m>n$, then both $X^{(m)}$ and $X^{(n)}$ solve the same stochastic equation as long as the state remains in $D_n$. Pathwise uniqueness therefore yields
$$ X_t^{(m)} = X_t^{(n)}, \qquad 0\leq t<\tau_n \quad\text{a.s.} $$
Hence the family $ \{X^{(n)}\}_{n\geq1} $ defines a pathwise unique maximal local solution $X$ on
$$ [0,\tau_\infty), \qquad \tau_\infty := \lim_{n\to\infty}\tau_n. $$
It remains to show that $$ \tau_\infty=\infty \qquad\text{a.s.}, $$ which is established below by deriving uniform moment bounds and ruling out explosion on finite time intervals.

\item\textit{Moment estimate for the Hawkes excitation.} We set
$$ \bar a_1 := \int_{\mathcal Z}a(\zeta)\,\nu(d\zeta), \qquad \bar a_2 := \int_{\mathcal Z}a(\zeta)^2\,\nu(d\zeta),
$$
which are finite by assumption \eqref{eq:int_a_a^2_finite}. On $[0,\tau_n)$, the localized process coincides with the original process and therefore
$$
dY_t = -bY_t\,dt + \int_{\mathcal Z}a(\zeta)\,N(dt,d\zeta).
$$
Applying It\^o's formula for semimartingales with jumps to
$f(y)=y^2$ yields
\begin{align}
	dY_t^2
	&= -2bY_t^2\,dt	+ \int_{\mathcal Z} \left[ (Y_{t-}+a(\zeta))^2-Y_{t-}^2 \right] N(dt,d\zeta) \notag\\
	&= -2bY_t^2\,dt + \int_{\mathcal Z} \left[ 2Y_{t-}a(\zeta)+a(\zeta)^2 \right] N(dt,d\zeta).
	\label{eq:Y2_WP}
\end{align}
For the It\^o formula for jump semimartingales, see, for example, 
\cite{protter_stochastic_2005} or \cite{applebaum_levy_2009}.

Stopping at $t\wedge\tau_n$, taking expectations, and applying the compensation formula for integer-valued random measures with predictable compensator \eqref{eq:predictable_compensator}, we obtain
\begin{align}
	\mathbb E[Y_{t\wedge\tau_n}^2]
	&= y^2 - 2b\, \mathbb E \int_0^{t\wedge\tau_n}Y_s^2\,ds + \mathbb E \int_0^{t\wedge\tau_n} \bigl(\mu(P_s)+Y_s\bigr) \bigl(2\bar a_1Y_s+\bar a_2\bigr)\,ds.
	\label{eq:Y2_exp_WP}
\end{align}
Here $P_{s-}$ since $P$ is continuous, while $Y_{s-}=Y_s$ for Lebesgue-a.e.\ $s$. By the linear-growth assumption on $\mu$, \eqref{eq:mu_growth_assumpt}, and Young's inequality, there exists a constant $A>0$, independent of $n$, such that
$$ \bigl(\mu(p)+y\bigr) \bigl(2\bar a_1y+\bar a_2\bigr) \leq A(1+p^2+y^2), \qquad p,y\geq0.
$$
Consequently,
\begin{equation}
	\mathbb E[Y_{t\wedge\tau_n}^2] \leq y^2 + A\int_0^t \left( 1 + \mathbb E[P_{s\wedge\tau_n}^2] + \mathbb E[Y_{s\wedge\tau_n}^2] \right)ds. \label{eq:Y2_bound_WP}
\end{equation}

For the subsequent nonexplosion argument it is convenient to obtain a maximal estimate. Since $Y\geq0$, the integral representation
\begin{equation} \label{eq:Y_integral_representation}
    Y_t = ye^{-bt} + \int_{(0,t]\times\mathcal Z} e^{-b(t-s)}a(\zeta)\,N(ds,d\zeta)
\end{equation}
implies, for $0\leq u\leq t\wedge\tau_n$,
$$ 
Y_u \leq y+ \int_{(0,u]\times\mathcal Z}a(\zeta)\,N(ds,d\zeta).
$$
Decomposing the last integral into its compensated martingale part and its compensator, and applying the Burkholder-Davis-Gundy inequality for purely discontinuous martingales together with the preceding moment bounds, we obtain, for some constant $A_T>0$ independent of $n$,
\begin{equation}
	\mathbb E \left[ \sup_{0\leq s\leq T\wedge\tau_n}Y_s^2 \right] \leq A_T \left[ 1+y^2 + \int_0^T \mathbb E \left( P_{s\wedge\tau_n}^2 + Y_{s\wedge\tau_n}^2 \right)ds \right].
	\label{eq:Y2_sup_WP}
\end{equation}

\item \textit{Moment estimates for capital and pollution.} We first consider the capital process. At a jump time $t$ with mark $\zeta$,
$$
K_t = \omega(K_{t-},P_{t-},\zeta)K_{t-}.
$$
Since $0<\omega(k,p,\zeta)\leq1$,
we have
$$
K_t^2-K_{t-}^2 = \left( \omega(K_{t-},P_{t-},\zeta)^2-1 \right)K_{t-}^2 \leq0.
$$
Thus disaster jumps cannot increase $K^2$. Applying the It\^o formula for jump semimartingales to $K^2$ on $[0,t\wedge\tau_n]$ therefore gives
\begin{align}
	K_{t\wedge\tau_n}^2
	&= k^2 + 2\int_0^{t\wedge\tau_n} K_s b^{\mathrm{cap}}(K_s,P_s,C_s,\theta_s)\,ds \notag\\
	&\quad + \int_{(0,t\wedge\tau_n]\times\mathcal Z} \left[ \omega(K_{s-},P_{s-},\zeta)^2-1 \right] K_{s-}^2\,N(ds,d\zeta) \notag\\
	&\leq k^2 + 2\int_0^{t\wedge\tau_n} K_s b^{\mathrm{cap}}(K_s,P_s,C_s,\theta_s)\,ds. 
	\label{eq:K2_ito_WP}
\end{align}

Using the linear-growth assumption on $b^{\mathrm{cap}}$ and Young's
inequality, there exists $A>0$, independent of $n$, such that
\begin{align}
	\mathbb E \left[ \sup_{0\leq u\leq t\wedge\tau_n}K_u^2 \right] \leq A \Bigg[ k^2 + \int_0^t \Bigg( 1 + \mathbb E \left[ \sup_{0\leq r\leq s\wedge\tau_n} (K_r^2+P_r^2) \right] + \mathbb E[C_s^2] \Bigg)ds \Bigg].
	\label{eq:K2_bound_WP}
\end{align}

For the pollution process, It\^o's formula yields
\begin{align}
	P_{t\wedge\tau_n}^2 = p^2 + 2\int_0^{t\wedge\tau_n} P_s b^{\mathrm{pol}}(K_s,P_s,C_s,\theta_s)\,ds + \sigma^2 \int_0^{t\wedge\tau_n}P_s^2\,ds + 2\sigma \int_0^{t\wedge\tau_n}P_s^2\,dW_s.
	\label{eq:P2_ito_WP}
\end{align}
Taking the supremum, expectations, and applying the Burkholder-Davis-Gundy inequality to the continuous local martingale term, followed by Young's inequality, gives
\begin{align}
	\mathbb E \left[ \sup_{0\leq u\leq t\wedge\tau_n}P_u^2 \right] \leq A \Bigg[ p^2 + \int_0^t \Bigg(1 + \mathbb E \left[ \sup_{0\leq r\leq s\wedge\tau_n} (K_r^2+P_r^2) \right] + \mathbb E[C_s^2] \Bigg)ds \Bigg],
	\label{eq:P2_bound_WP}
\end{align}
where $A>0$ may depend on $T$, $\sigma$, and the growth constants,
but not on $n$. For the Burkholder-Davis-Gundy inequality, see, for example, \cite{revuz_yor_1999}.
 
Combining \eqref{eq:Y2_sup_WP}, \eqref{eq:K2_bound_WP}, and \eqref{eq:P2_bound_WP}, and applying Gronwall's inequality, we obtain
\begin{equation}
	\mathbb E \left[ \sup_{0\leq t\leq T\wedge\tau_n} \left( K_t^2+P_t^2+Y_t^2 \right) \right] \leq A_T \left( 1+k^2+p^2+y^2 + \mathbb E\int_0^T C_t^2\,dt \right),
	\label{eq:joint_bound_WP}
\end{equation}
where $A_T<\infty$ is independent of $n$.
	
\item \textit{Nonexplosion, preservation of the state space, and global continuation.} Fix $T<\infty$. We first show that the marked point process cannot have infinitely many jumps before maximal lifetime. For every $n\in\mathbb N$, the compensation formula applied to the stopped counting measure yields
	\begin{align}
		\mathbb E \left[ N((0,T\wedge\tau_n]\times\mathcal Z) \right]
		&= \mathbb E \int_0^{T\wedge\tau_n} \bigl( \mu(P_{t-})+Y_{t-} \bigr)\,dt \notag\\
		&= \mathbb E \int_0^{T\wedge\tau_n} \bigl( \mu(P_t)+Y_t \bigr)\,dt.
		\label{eq:stopped_expected_number_jumps}
	\end{align}
By the linear-growth assumption on $\mu$ \eqref{eq:mu_growth_assumpt}, Cauchy-Schwarz, and \eqref{eq:joint_bound_WP}, there exists $B_T<\infty$, independent of $n$, such that
\begin{equation}
	\sup_{n\geq1} \mathbb E \left[ N((0,T\wedge\tau_n]\times\mathcal Z) \right] \leq B_T.
	\label{eq:uniform_jump_bound}
\end{equation}
Since $\tau_n\uparrow\tau_\infty$, the random variables $N((0,T\wedge\tau_n]\times\mathcal Z)$ increase to $N((0,T\wedge\tau_\infty)\times\mathcal Z)$.
Therefore, by the monotone convergence theorem,
\begin{align}
	\mathbb E \left[ N((0,T\wedge\tau_\infty)\times\mathcal Z) \right] = \lim_{n\to\infty} \mathbb E \left[ N((0,T\wedge\tau_n]\times\mathcal Z) \right] \leq B_T <\infty.
	\label{eq:jumps_before_lifetime}
\end{align}
Consequently,
\begin{equation}
	N((0,T\wedge\tau_\infty)\times\mathcal Z) <\infty \qquad\text{a.s.}
	\label{eq:finitely_many_jumps_before_lifetime}
\end{equation}
Thus there can be no accumulation of disaster times strictly before,
or approaching, a finite maximal lifetime.

We next rule out escape of the state variables to $+\infty$. For $R>0$, we define
$$
\sigma_R := \inf\left\{ t<\tau_\infty: K_t+P_t+Y_t\geq R \right\}.
$$
Then, for every $n$ sufficiently large relative to $R$,
$$
\{\sigma_R\leq T,\ \sigma_R<\tau_n\} \subseteq \left\{ \sup_{0\leq t\leq T\wedge\tau_n} (K_t+P_t+Y_t)\geq R \right\}.
$$
Hence, by Markov's inequality,
\begin{align}
	\mathbb P(\sigma_R\leq T,\ \sigma_R<\tau_n)
	&\leq \frac{1}{R^2} \mathbb E \left[ \sup_{0\leq t\leq T\wedge\tau_n} (K_t+P_t+Y_t)^2 \right] \notag\\
	&\leq \frac{3}{R^2} \mathbb E \left[ \sup_{0\leq t\leq T\wedge\tau_n} (K_t^2+P_t^2+Y_t^2)
	\right].
\end{align}
Using \eqref{eq:joint_bound_WP} and then letting $n\to\infty$ gives
\begin{equation}
	\mathbb P(\sigma_R\leq T) \leq \frac{3A_T}{R^2}\left( 1+k^2+p^2+y^2 + \mathbb E\int_0^T C_t^2\,dt \right).
	\label{eq:upper_explosion_probability}
\end{equation}
Letting $R\to\infty$ yields
$$
\mathbb P \left( \sup_{0\leq t<T\wedge\tau_\infty} (K_t+P_t+Y_t)=\infty \right)=0.
$$
Thus the maximal local solution cannot cease to exist because of
escape to $+\infty$ in finite time. It remains to exclude a finite exit through the lower boundary of $\mathcal S_H$. By \eqref{eq:finitely_many_jumps_before_lifetime}, on every finite time interval before $\tau_\infty$ there are almost surely only finitely many disaster times. At every such disaster time, $ K_t = \omega(K_{t-},P_{t-},\zeta)K_{t-}$,
and the assumption \eqref{eq:w_at_disaster} implies that a strictly positive value of $K$ remains strictly positive after the jump. The process $P$ does not jump.

Suppose now, for contradiction, that for some $T <\infty$
$$
\mathbb P(\tau_\infty\leq T)>0.
$$
On the event $\{\tau_\infty\leq T\}$, escape to $+\infty$ has already been ruled out, and \eqref{eq:finitely_many_jumps_before_lifetime} rules out accumulation of disaster times. Hence the only remaining possible obstruction to continuation would be attainment of the lower boundary $K=0$ or $P=0$. This is excluded by Assumption~\ref{Assumptions}(e) and the fact that each jump preserves strict positivity. We obtain a contradiction.
Therefore,
\begin{equation}
	\tau_\infty=\infty \qquad\text{a.s.}
	\label{eq:global_lifetime}
\end{equation}
Hence the maximal local solution extends uniquely to a strong solution on the whole interval $[0,\infty)$.

We may now remove the stopping from the jump-count estimate. Since $\tau_\infty=\infty$ a.s., for every finite $T$,
$$
T\wedge\tau_n\uparrow T \qquad\text{a.s.}
$$
and therefore, by monotone convergence,
\begin{align}
	\mathbb E \left[ N((0,T]\times\mathcal Z) \right] = \mathbb E \left[ \int_0^T \bigl(\mu(P_t)+Y_t\bigr)\,dt \right] <\infty.
	\label{eq:expected_number_jumps}
\end{align}
Since $N((0,T]\times\mathcal Z)$ is a nonnegative integer-valued random variable, possibly taking the value $+\infty$, finite expectation implies
$$
N((0,T]\times\mathcal Z)<\infty \qquad\text{a.s.}
$$
Thus the marked point process is locally finite and nonexplosive.

Finally, global pathwise uniqueness follows from pathwise uniqueness of the localized equations. Indeed, if $X$ and $\widetilde X$ are two global strong solutions driven by the same Brownian motion and the same underlying Poisson random measure with the same initial condition and control, then for every $n$ their restrictions up to the corresponding localization time solve the same localized equation. Pathwise uniqueness gives equality up to each localization time, and since the localization times increase to infinity almost surely,
$$
X_t=\widetilde X_t \qquad\text{for every }t\geq0, \quad\text{a.s.}
$$

Finally, since $\tau_\infty=\infty$ a.s., Fatou's lemma applied to \eqref{eq:joint_bound_WP} gives, for every $T<\infty$,
\begin{align}
	\mathbb E \left[ \sup_{0\leq t\leq T} (K_t^2+P_t^2+Y_t^2) \right] \leq A_T \left( 1+k^2+p^2+y^2 + \mathbb E\int_0^T C_t^2\,dt \right) <\infty.
	\label{eq:global_finite_horizon_moment}
\end{align}
Consequently,
\begin{equation} \label{eq:local_square_integrability}
\mathbb E \left[ \int_0^T (K_t^2+P_t^2+Y_t^2)\,dt \right] \leq T\, \mathbb E \left[ \sup_{0\leq t\leq T} (K_t^2+P_t^2+Y_t^2) \right] <\infty.
\end{equation}
Thus the global solution is locally square-integrable in time.
\end{enumerate}
\end{proof}

\begin{remark}
The term \emph{global well-posedness} means that the system admits a pathwise unique strong solution defined on the whole time interval $[0,\infty)$; equivalently, the maximal lifetime satisfies $\tau_\infty=\infty$ a.s. This should be distinguished from global square integrability of the state process. The moment estimates established in \eqref{eq:local_square_integrability} imply local square integrability in time, but no claim is made here that
$$
\mathbb E\left[ \int_0^\infty \left(K_t^2+P_t^2+Y_t^2\right)\,dt \right] <\infty,
$$
which is an additional integrability property and is not part of the global well-posedness assertion.
\end{remark}

Having established that the Hawkes-driven system admits a unique global strong solution, we now turn to the relation between the Hawkes and
Poisson specifications. Recall that the Hawkes intensity can be decomposed as in \eqref{eq:lambda_Hawkes}. Thus the only difference at the level of the disaster intensity is the additional excitation term $Y$. It is therefore natural to first study the long-run behaviour of this component. If $Y$ becomes negligible, then the Hawkes intensity approaches the pollution-dependent Poisson intensity $\mu(P_{t-})$.

\subsection{Asymptotic Poisson approximation}

\n 

The purpose of this subsection is to make the preceding intuition precise. We first establish decay and integrability properties of the excitation process. These properties are then translated into approximation results for the compensator of the marked point process, first on finite time windows and subsequently on the infinite tail.

\begin{theorem}[Asymptotic decay of the Hawkes excitation]
	\label{thm:asymptotic_decay_marked_hawkes}
Assume the hypotheses of Proposition~\ref{prop:wellposed_marked_hawkes}. Set
	$$
	\bar a := \int_{\mathcal Z}a(\zeta)\,\nu(d\zeta),
	$$
and suppose that the marked Hawkes process is subcritical, $\bar a<b$. Assume further that
	\begin{equation}
	\int_0^\infty \mathbb E[\mu(P_t)]\,dt <\infty.
    \label{eq:expect_mu_P_finite}
	\end{equation}
Then
\begin{equation}
\mathbb E[Y_t] \longrightarrow 0 \qquad\text{as }  t \to\infty.
\label{eq:expect_Y_to_zero}
\end{equation}
Moreover,
	\begin{equation}
	\int_0^\infty \mathbb E[Y_t]\,dt = \frac{y + \bar a \displaystyle\int_0^\infty \mathbb E[\mu(P_t)]\,dt}{ b-\bar a } <\infty.
    \label{eq:expect_Y_finite}
	\end{equation}
In particular, $Y_t \longrightarrow 0$ in $L^1$, i.e. the Hawkes self-excitation vanishes asymptotically in $L^1$. However, unless $Y_0=0$ and no jump with positive excitation ever occurs, the process $Y_t$ does not reach zero at any finite time.
\end{theorem}

\begin{proof}
From \eqref{eq:WP_Y}, we obtain, after integration,
	$$
     Y_t = y - b\int_0^tY_s\,ds + \int_{(0,t]\times\mathcal Z} a(\zeta)\,N(ds,d\zeta).
	$$
Taking expectations and applying the compensation formula gives
	\begin{align*}
		\mathbb E[Y_t]
		&= y - b\int_0^t\mathbb E[Y_s]\,ds + \mathbb E \int_0^t\int_{\mathcal Z} a(\zeta)\, \bigl(\mu(P_{s-})+Y_{s-}\bigr) \nu(d\zeta)\,ds.
	\end{align*}
	Since $P$ is continuous and $Y_{s-}=Y_s$ for Lebesgue-a.e. $s$,
	this becomes
	\begin{align*}
		\mathbb E[Y_t]
		&= y - b\int_0^t\mathbb E[Y_s]\,ds + \bar a \int_0^t \mathbb E\bigl[\mu(P_s)+Y_s\bigr]\,ds \\
		&= y + \int_0^t \left[ -(b-\bar a)\mathbb E[Y_s] + \bar a\,\mathbb E[\mu(P_s)] \right]ds.
	\end{align*}
Set $m(t):=\mathbb E[Y_t]$. Then $m$ is absolutely continuous and satisfies, for almost every $t\ge0$,
\begin{equation}
m'(t) = -(b-\bar a)m(t) + \bar a\,\mathbb E[\mu(P_t)], \qquad m(0)=y.
\label{eq:m_differential_equation}
\end{equation}
By the variation-of-constants formula,
$$
	m(t) = e^{-(b-\bar a)t}y + \bar a \int_0^t e^{-(b-\bar a)(t-s)} \mathbb E[\mu(P_s)]\,ds. 
$$
	
Let
	$$
f(s):=\mathbb E[\mu(P_s)], \qquad g(s):=e^{-(b-\bar a)s}\mathbf 1_{\{s\ge0\}}.
	$$
By assumption, $f\in L^1(\mathbb R_{\ge0})$, and, since $b-\bar a>0$,
	$
	g\in L^1(\mathbb R_{\ge0}).
	$
Hence
	$$
	(g*f)(t) = \int_0^t e^{-(b-\bar a)(t-s)} f(s)\,ds \longrightarrow 0 \qquad\text{as }t\to\infty.
	$$
Therefore
	\begin{equation} \label{eq:m(t)_convergence_to_0}
	m(t)=\mathbb E[Y_t]\longrightarrow0 \qquad\text{as }t\to\infty.
	\end{equation}
Integrating \eqref{eq:m_differential_equation} over $[0,T]$ gives
	$$
	m(T)-y = -(b-\bar a)\int_0^T m(t)\,dt + \bar a \int_0^T\mathbb E[\mu(P_t)]\,dt.
	$$
Letting $T\to\infty$ and using \eqref{eq:m(t)_convergence_to_0}, we obtain
	$$
	\int_0^\infty m(t)\,dt = \frac{ y + \bar a\displaystyle\int_0^\infty\mathbb E[\mu(P_t)]\,dt}{b-\bar a}.
	$$
	Since $Y_t\ge0$,
	$$
    \|Y_t\|_{L^1} = \mathbb E[Y_t] \longrightarrow 0,
	$$
which proves the $L^1$-convergence.
	
Finally, using the integral representation \eqref{eq:Y_integral_representation}, all terms on the right-hand side are nonnegative. If $y>0$, then
	$
	Y_t\ge ye^{-bt}>0
	$
for every $t < \infty$. Likewise, if a jump occurs at time $\tau$ with mark $\zeta$ such that $a(\zeta)>0$, then
	$$
	Y_t \geq e^{-b(t-\tau)}a(\zeta)>0 \qquad \text{for every finite }t>\tau.
	$$
Thus finite-time extinction of $Y$ does not occur under the exponential-kernel specification.
\end{proof}

The preceding theorem gives the basic asymptotic estimate required for the approximation argument. In particular, the excitation vanishes in mean and its expected cumulative contribution over the infinite horizon is finite. Two immediate consequences are useful. First, the $L^1$ convergence implies convergence in probability. Second, since $Y$ is exactly the difference between the Hawkes intensity and its Poisson baseline, the same estimate directly controls the excess disaster intensity.

\begin{corollary}[Convergence in probability]
	Under the assumptions of Theorem~\ref{thm:asymptotic_decay_marked_hawkes},
	$$
	Y_t\longrightarrow0 \qquad\text{in probability as }t\to\infty.
	$$
\end{corollary}

\begin{proof}
Since $Y_t\geq0$, Markov's inequality yields, for every $\varepsilon>0$,
	$$
	\mathbb P\bigl(|Y_t|>\varepsilon\bigr) = \mathbb P(Y_t>\varepsilon) \leq \frac{\mathbb E[Y_t]}{\varepsilon}.
	$$
By Theorem~\ref{thm:asymptotic_decay_marked_hawkes},
	$$
	\mathbb E[Y_t]\longrightarrow 0 \qquad \text{as }t \to \infty,
	$$
and therefore
	$$
	\mathbb P(Y_t>\varepsilon)\longrightarrow 0 \qquad\text{for every } \varepsilon>0,.
	$$
\end{proof}

Convergence of the state variable $Y$ also has a direct interpretation in terms of the disaster mechanism. Indeed, $Y_{t-}$ is precisely the self-exciting component of the Hawkes intensity. The following corollary therefore reformulates the preceding asymptotic result at the level of the stochastic intensity itself.

\begin{corollary}[Asymptotic decay of the excess disaster intensity] \label{cor:excess_intensity_decay}
Under the assumptions of Theorem~\ref{thm:asymptotic_decay_marked_hawkes},
	$$
	\lambda_t-\mu(P_{t-})=Y_{t-},
	$$
	and consequently
	$$
	\mathbb E \left[ \left| \lambda_t-\mu(P_{t-}) \right| \right] \longrightarrow 0 \qquad\text{as }t\to\infty.
	$$
Thus the excess disaster intensity generated by Hawkes self-excitation vanishes asymptotically in $L^1$.
\end{corollary}

\begin{proof}
Since $Y_{t-}\geq0$, it follows that
	$$
	\left| \lambda_t-\mu(P_{t-}) \right| = Y_{t-},
	$$
and hence
	\begin{equation}
		\mathbb E\left[ \left| \lambda_t-\mu(P_{t-}) \right|\right] = \mathbb E[Y_{t-}].
		\label{eq:excess_intensity_Y}
	\end{equation}
Since
	$$
	\Delta Y_t = \int_{\mathcal Z} a(\zeta)\,N(\{t\},d\zeta),
	$$
we have $Y_t-Y_{t-}=\Delta Y_t$. Moreover, the marked point process $N$ has predictable compensator \eqref{eq:predictable_compensator}, which is absolutely continuous with respect to Lebesgue measure in time. Consequently, for every fixed deterministic $t>0$,
	$$
	\mathbb P \left( N(\{t\}\times\mathcal Z)>0 \right) =0.
    $$
Thus $\Delta Y_t=0$ a.s. for every fixed $t$, and therefore $Y_{t-}=Y_t$ a.s.
In particular, $ \mathbb E[Y_{t-}] = \mathbb E[Y_t]$. Combining the result \eqref{eq:expect_Y_to_zero} of Theorem~\ref{thm:asymptotic_decay_marked_hawkes} with \eqref{eq:excess_intensity_Y}, we obtain
	$$
	\mathbb E\left[ \left| \lambda_t-\mu(P_{t-}) \right| \right] = \mathbb E[Y_t] \longrightarrow 0 \qquad\text{as } t \to \infty.
	$$
Hence
	$$
	\lambda_t-\mu(P_{t-}) \longrightarrow 0 \qquad \text{in } L^1,
	$$
which proves that the excess disaster intensity generated by Hawkes self-excitation vanishes asymptotically.
\end{proof}

The previous results concern the pointwise difference between the Hawkes and Poisson intensities. For the subsequent control argument, it is useful to translate this pointwise decay into a statement about the corresponding random measures. We therefore compare the expected integrals of a test function against the Hawkes counting measure with those obtained from the pollution-dependent Poisson compensator.

We first perform this comparison on a fixed time window of length $H>0$. This formulation isolates the contribution of the excitation term and provides an explicit approximation error.

\begin{proposition}[Finite-window Poisson approximation of the Hawkes intensity]	\label{prop:finite_window_poisson_approximation}

Assume the hypotheses of Theorem~\ref{thm:asymptotic_decay_marked_hawkes}, so that
	$$
	\mathbb E[Y_t]\longrightarrow 0 \qquad\text{as }t\to\infty.
	$$
Let $H>0$ be fixed and let
	$$
	\phi: \mathbb R_{\geq0} \times \mathcal Z \longrightarrow\mathbb R
	$$
be continuous and bounded. Then
	\begin{align}
		&\left| \mathbb E\left[ \int_{(T,T+H]\times\mathcal Z} \phi(t,\zeta)\,N(dt,d\zeta) \right] - \mathbb E\left[ \int_T^{T+H}\int_{\mathcal Z} \phi(t,\zeta)\mu(P_{t-})\,\nu(d\zeta)\,dt \right] \right| \notag\\
		&\qquad\leq \|\phi\|_\infty \int_T^{T+H}\mathbb E[Y_t]\,dt.
		\label{eq:finite_window_poisson_error}
	\end{align}
	
Consequently,
	$$
	\lim_{T\to\infty} \left| \mathbb E\left[ \int_{(T,T+H]\times\mathcal Z} \phi(t,\zeta)\,N(dt,d\zeta) \right] - \mathbb E\left[ \int_T^{T+H}\int_{\mathcal Z} \phi(t,\zeta)\mu(P_{t-})\,\nu(d\zeta)\,dt \right] \right| =0.
	$$
	
More precisely, for every $\varepsilon>0$, there exists $T_\varepsilon<\infty$ such that, for every $T\geq T_\varepsilon$,
	$$
	\left| \mathbb E\left[ \int_{(T,T+H]\times\mathcal Z} \phi(t,\zeta)\,N(dt,d\zeta) \right] -
	\mathbb E\left[ \int_T^{T+H}\int_{\mathcal Z} \phi(t,\zeta)\mu(P_{t-})\,\nu(d\zeta)\,dt \right] \right| \leq \|\phi\|_\infty H\varepsilon.
	$$
\end{proposition}

\begin{proof}
By the compensation formula and the predictable compensator \eqref{eq:predictable_compensator}, we have
	\begin{align*}
		\mathbb E\left[ \int_{(T,T+H]\times\mathcal Z} \phi(t,\zeta)\,N(dt,d\zeta) \right]
		&= \mathbb E\left[ \int_T^{T+H}\int_{\mathcal Z} \phi(t,\zeta) \bigl(\mu(P_{t-})+Y_{t-}\bigr) \nu(d\zeta)\,dt \right].
	\end{align*}
Subtracting the pollution-dependent baseline contribution gives
	\begin{align*}
		&\mathbb E\left[ \int_{(T,T+H]\times\mathcal Z} \phi(t,\zeta)\,N(dt,d\zeta) \right] - \mathbb E\left[ \int_T^{T+H}\int_{\mathcal Z} \phi(t,\zeta)\mu(P_{t-})\,\nu(d\zeta)\,dt \right] \\
		&\qquad= \mathbb E\left[ \int_T^{T+H}\int_{\mathcal Z} \phi(t,\zeta)Y_{t-}\, \nu(d\zeta)\,dt \right] \\
		&\qquad \leq \mathbb E\left[ \int_T^{T+H}\int_{\mathcal Z} |\phi(t,\zeta)|Y_{t-}\, \nu(d\zeta)\,dt \right] \leq \|\phi\|_\infty \int_T^{T+H} \mathbb E[Y_t]\,dt,
	\end{align*}
where we used $\nu(\mathcal Z)=1$ and the fact that $Y_{t-}=Y_t$ for Lebesgue-a.e.\ $t$.
	
Since
	$
	\mathbb E[Y_t]\longrightarrow0,
	$
for every $\varepsilon>0$ there exists $T_\varepsilon$ such that
	$$
	\mathbb E[Y_t]\leq\varepsilon, \qquad t \geq T_\varepsilon.
	$$
Thus, for every $T\geq T_\varepsilon$,
	$$
	\int_T^{T+H}\mathbb E[Y_t]\,dt \leq H \varepsilon,
	$$
which proves the claim.
\end{proof}
The finite-window estimate shows that, sufficiently far in the future, the expected contribution of Hawkes self-excitation over any fixed time interval becomes arbitrarily small. For the infinite-horizon control problem, however, a tail estimate is more directly relevant. The integrability established in Theorem~\ref{thm:asymptotic_decay_marked_hawkes} allows the preceding argument to be extended from finite windows to the entire remaining time horizon.

\begin{corollary}[Approximation on the infinite tail]
\label{cor:infinite_tail_approximation}
Under the assumptions of Theorem~\ref{thm:asymptotic_decay_marked_hawkes}, let
	$$
	\phi:\mathbb R_{\geq0}\times\mathcal Z\to\mathbb R
	$$
be continuous and bounded. Then, for every $T\geq0$,
	\begin{align}
		&\left| \mathbb E\left[ \int_{(T,\infty)\times\mathcal Z} \phi(t,\zeta)\,N(dt,d\zeta) \right] - \mathbb E\left[ \int_T^\infty\int_{\mathcal Z} \phi(t,\zeta)\mu(P_{t-})\,\nu(d\zeta)\,dt \right] \right| \leq \|\phi\|_\infty \int_T^\infty\mathbb E[Y_t]\,dt.
		\label{eq:infinite_tail_approximation}
	\end{align}
Consequently,
	$$
	\left| \mathbb E\left[ \int_{(T,\infty)\times\mathcal Z} \phi(t,\zeta)\,N(dt,d\zeta) \right] - \mathbb E\left[ \int_T^\infty\int_{\mathcal Z} \phi(t,\zeta)\mu(P_{t-})\, \nu(d\zeta)\,dt \right] \right| \longrightarrow0
	$$
as $T\to\infty$.
\end{corollary}

\begin{proof}
Since $\phi$ is continuous and bounded, it is Borel measurable, and hence the deterministic map
	$$
    (\omega,t,\zeta) \longmapsto \phi(t,\zeta)
	$$
defines an admissible predictable integrand for the compensation formula. Moreover, by the assumption of Theorem~\ref{thm:asymptotic_decay_marked_hawkes}, \eqref{eq:expect_mu_P_finite}, and \eqref{eq:expect_Y_finite},
	\begin{align*}
		&\mathbb E\left[ \int_T^\infty\int_{\mathcal Z} |\phi(t,\zeta)| \bigl(\mu(P_{t-})+Y_{t-}\bigr) \nu(d\zeta)\,dt \right] \\
		&\qquad \leq \|\phi\|_\infty \left( \int_T^\infty \mathbb E[\mu(P_t)]\,dt + \int_T^\infty	\mathbb E[Y_t]\,dt \right) < \infty,
	\end{align*}
where we used $\nu(\mathcal Z)=1$ and the fact that $P_{t-}=P_t$ and $Y_{t-}=Y_t$ for Lebesgue-a.e.\ $t$. Thus, the compensation formula yields
	\begin{align}
		\mathbb E\left[\int_{(T,\infty)\times\mathcal Z}	\phi(t,\zeta)\,N(dt,d\zeta) \right]
		&= \mathbb E\left[ \int_T^\infty\int_{\mathcal Z} \phi(t,\zeta) \bigl(\mu(P_{t-})+Y_{t-}\bigr) \nu(d\zeta)\,dt \right].
		\label{eq:tail_compensation}
	\end{align}
Subtracting the pollution-dependent baseline contribution and taking absolute values, we obtain
\begin{align*}
&\left| \mathbb E\left[ \int_{(T,\infty)\times\mathcal Z} \phi(t,\zeta)\,N(dt,d\zeta) \right] - \mathbb E\left[ \int_T^\infty\int_{\mathcal Z} \phi(t,\zeta)\mu(P_{t-})\,\nu(d\zeta)\,dt \right] \right|
\\
&\qquad \leq \mathbb E\left[ \int_T^\infty\int_{\mathcal Z} |\phi(t,\zeta)|Y_{t-}\,\nu(d\zeta)\,dt \right]
\\
&\qquad \leq \|\phi\|_\infty \int_T^\infty\mathbb E[Y_t]\,dt,
\end{align*}
which proves \eqref{eq:infinite_tail_approximation}.

Because $Y_{t-}\geq0$, Tonelli's theorem yields
\begin{align*}
\mathbb E\left[ \int_T^\infty\int_{\mathcal Z} Y_{t-}\,\nu(d\zeta)\,dt \right]
&= \int_T^\infty\int_{\mathcal Z} \mathbb E[Y_{t-}]\,\nu(d\zeta)\,dt = \int_T^\infty \mathbb E[Y_{t-}]\,dt,
\end{align*}
where we used $\nu(\mathcal Z)=1$. Since $Y$ is c\`adl\`ag, we have $Y_{t-}=Y_t$ for Lebesgue-a.e.\ $t$, almost surely. Therefore,
$$
\int_T^\infty \mathbb E[Y_{t-}]\,dt = \int_T^\infty \mathbb E[Y_t]\,dt.
$$
Combining the preceding estimates proves \eqref{eq:infinite_tail_approximation}.

Finally, by \eqref{eq:expect_Y_finite}, the nonnegative function $t\mapsto\mathbb E[Y_t]$ is integrable on $\mathbb R_{\geq0}$. Hence
$$
\int_T^\infty\mathbb E[Y_t]\,dt \longrightarrow0 \qquad\text{as }T\to\infty.
$$
Therefore, by \eqref{eq:infinite_tail_approximation}, the contribution of the Hawkes self-excitation on the infinite tail vanishes as $T\to\infty$.
\end{proof}

\subsection{Small-excitation approximation of the value function}

\n 

The preceding results describe an asymptotic approximation in time: under suitable integrability conditions, the self-exciting component of the Hawkes intensity becomes negligible in the distant future. This observation alone, however, does not yield a uniform estimate for the difference between the corresponding infinite-horizon value functions. To obtain such an estimate, we now introduce a small-excitation regime. More precisely, we scale the excitation generated by each disaster mark according to
$$
    a_\eta(\zeta)=\eta a(\zeta), \qquad \eta>0,
$$
and study the limit $\eta\downarrow0$. In this regime, the Hawkes model approaches the Poisson model directly at the level of the controlled generator. The remaining task is to control the cumulative contribution of the excitation term uniformly over admissible controls.

The first step is to obtain a quantitative estimate for the excitation process under this scaling. Since the generator comparison below involves the product of $Y^\eta$ with the jump increment of the Poisson value function, an unweighted estimate for $Y^\eta$ alone is not sufficient. We therefore establish a weighted discounted bound that is uniform both in $\eta$ and over the admissible controls.

\begin{proposition}[Uniform weighted small-excitation estimate]
	\label{prop:small_excitation_weighted_estimate}
	Let $\eta_0>0$ and, for each $\eta\in(0,\eta_0]$, define
	$$
	a_\eta(\zeta):=\eta a(\zeta), \qquad \zeta\in\mathcal Z.
	$$
	Let
	$$
	X^{H,\eta} = (K^{H,\eta},P^{H,\eta},Y^\eta)
	$$
	denote the corresponding controlled Hawkes state process, whose excitation component satisfies
	\begin{equation}
		\label{eq:Y_eta_dynamics}
		dY_t^\eta = -bY_t^\eta\,dt + \int_{\mathcal Z}
		\eta a(\zeta)\, N^\eta(dt,d\zeta),
		\qquad
		Y_0^\eta=y_\eta\geq0,
	\end{equation}
	where $N^\eta$ has predictable compensator
	\begin{equation}
		\label{eq:N_eta_compensator}
		\bigl(
		\mu(P_{t-}^{H,\eta})+Y_{t-}^\eta
		\bigr)
		\nu(d\zeta)\,dt.
	\end{equation}
	Set
	$$
	 \bar a_1 := \int_{\mathcal Z}a(\zeta)\,\nu(d\zeta) < \infty, \qquad \bar a_2 := \int_{\mathcal Z}a(\zeta)^2\,\nu(d\zeta) < \infty.
	$$
	Assume furthermore that
	\begin{enumerate}
		
		\item 
		There exists $c_y>0$, independent of $\eta$, such that
		\begin{equation}
			\label{eq:y_eta_scaling}
			0\leq y_\eta\leq c_y\eta,
			\qquad
			0<\eta\leq\eta_0.
		\end{equation}
		
		\item
		The family is uniformly subcritical:
		\begin{equation}
			\label{eq:uniform_subcritical_eta}
			\eta_0\bar a_1<b.
		\end{equation}
		
		\item
		There exists $B_\mu<\infty$ such that
		\begin{equation}
			\label{eq:uniform_discounted_mu_moment}
			\sup_{0<\eta\leq\eta_0}
			\sup_{\alpha\in\mathcal A_H(k,p,y_\eta)}
			\mathbb E_{k,p,y_\eta}^{\alpha}
			\left[ \int_0^\infty e^{-\rho t} \mu(P_t^{H,\eta})^2\,dt \right] \leq B_\mu.
		\end{equation}
		
		\item
		For some $q\geq0$, there exists $B_q<\infty$ such that
		\begin{equation}
			\label{eq:uniform_discounted_state_moment}
			\sup_{0<\eta\leq\eta_0}
			\sup_{\alpha\in\mathcal A_H(k,p,y_\eta)}
			\mathbb E_{k,p,y_\eta}^{\alpha}
			\left[ \int_0^\infty e^{-\rho t} \left( 1+ (K_t^{H,\eta})^{2q} + (P_t^{H,\eta})^{2q} \right)dt \right] \leq B_q.
		\end{equation}
		
	\end{enumerate}
Then there exists a constant $A_*>0$, independent of
	$\eta\in(0,\eta_0]$ and of the admissible control, such that
	\begin{equation}
		\label{eq:weighted_excitation_eta_bound}
		\sup_{\alpha\in\mathcal A_H(k,p,y_\eta)}
		\mathbb E_{k,p,y_\eta}^{\alpha}
		\left[ \int_0^\infty e^{-\rho t} Y_t^\eta \left( 1+ (K_t^{H,\eta})^q + (P_t^{H,\eta})^q \right)dt \right] \leq A_*\eta.
	\end{equation}
In particular,
	$$ \sup_{\alpha\in\mathcal A_H(k,p,y_\eta)} \mathbb E_{k,p,y_\eta}^{\alpha} \left[ \int_0^\infty e^{-\rho t} Y_t^\eta \left(1+ (K_t^{H,\eta})^q + (P_t^{H,\eta})^q \right)dt \right] \longrightarrow 0 \qquad \text{ as } \eta\downarrow 0.
	$$
\end{proposition}

\begin{proof}
Fix $\eta\in(0,\eta_0]$ and $\alpha\in\mathcal A_H(k,p,y_\eta)$.
Let us consider the proof divided in two steps:
\begin{enumerate}
	\item \textit{Uniform discounted $L^2$-estimate for the Hawkes excitation.} For notational convenience, throughout this step we write $P_t^\eta:=P_t^{H,\eta}$. The excitation process satisfies
		\begin{equation}
			\label{eq:Y_eta_dynamics_proof}
			dY_t^\eta = -bY_t^\eta\,dt + \int_{\mathcal Z} \eta a(\zeta)\,N^\eta(dt,d\zeta), \qquad Y_0^\eta=y_\eta.
		\end{equation}
	Applying It\^o's formula for jump semimartingales to $f(y)=y^2$ gives
		\begin{align}
			d(Y_t^\eta)^2
			&= -2b(Y_t^\eta)^2\,dt + \int_{\mathcal Z} \left[ \bigl(Y_{t-}^\eta+\eta a(\zeta)\bigr)^2 - \bigl(Y_{t-}^\eta\bigr)^2 \right] N^\eta(dt,d\zeta)\notag\\
			&= -2b(Y_t^\eta)^2\,dt + \int_{\mathcal Z} \left[ 2\eta Y_{t-}^\eta a(\zeta) + \eta^2a(\zeta)^2 \right]N^\eta(dt,d\zeta).
			\label{eq:Y_eta_square}
		\end{align}
    Since $t\mapsto e^{-\rho t}$ is continuous and of finite variation,
		the product rule yields
		\begin{align}
			d\left( e^{-\rho t}(Y_t^\eta)^2
			\right) &= -(2b+\rho)e^{-\rho t}(Y_t^\eta)^2\,dt + e^{-\rho t} \int_{\mathcal Z} \left[ 2\eta Y_{t}^\eta a(\zeta) + \eta^2a(\zeta)^2 \right] N^\eta(dt,d\zeta).
			\label{eq:discounted_Y_eta_square}
		\end{align}
	We now justify the compensation step by localization. Let
	$(\tau_n)_{n\geq1}$ be an increasing localizing sequence such that $\tau_n\uparrow\infty$ a.s. and the relevant stopped stochastic integrals are true martingales. Applying \eqref{eq:discounted_Y_eta_square} on $[0,T\wedge\tau_n]$ yields
		\begin{align}
			e^{-\rho(T\wedge\tau_n)} \bigl(Y_{T\wedge\tau_n}^\eta\bigr)^2 - y_\eta^2 &= -(2b+\rho) \int_0^{T\wedge\tau_n} e^{-\rho t}(Y_t^\eta)^2\,dt \notag\\
			&\quad + \int_{(0,T\wedge\tau_n]\times\mathcal Z} e^{-\rho t} \left[2\eta Y_{t-}^\eta a(\zeta) + \eta^2a(\zeta)^2 \right] N^\eta(dt,d\zeta).
			\label{eq:Y_eta_square_localized}
		\end{align}
	Hence, by the compensation formula,
		\begin{align}
			&\mathbb E_{k,p,y_\eta}^{\alpha} \left[e^{-\rho(T\wedge\tau_n)} \bigl(Y_{T\wedge\tau_n}^\eta\bigr)^2 \right] - y_\eta^2
			\notag\\
			&= -(2b+\rho)\mathbb E_{k,p,y_\eta}^{\alpha} \left[ \int_0^{T\wedge\tau_n} e^{-\rho t}(Y_t^\eta)^2\,dt \right]
			\notag\\
			&\quad+ \mathbb E_{k,p,y_\eta}^{\alpha} \left[ \int_0^{T\wedge\tau_n} e^{-\rho t} \bigl( \mu(P_t^\eta)+Y_t^\eta \bigr) \left( 2\eta\bar a_1Y_t^\eta + \eta^2\bar a_2 \right)dt \right].
			\label{eq:Y_eta_squared_expectation_localized}
		\end{align}
	Here we have used the continuity of $P^\eta$ and the fact that a c\`adl\`ag process coincides with its left-limit process for Lebesgue-a.e.\ time.
	
	We next estimate the compensator integrand. Expanding, we obtain
		\begin{align}
			\bigl( \mu(P_t^\eta)+Y_t^\eta \bigr) \left( 2\eta\bar a_1Y_t^\eta + \eta^2\bar a_2 \right) =
			2\eta\bar a_1 \mu(P_t^\eta)Y_t^\eta + 2\eta\bar a_1(Y_t^\eta)^2 + \eta^2\bar a_2\mu(P_t^\eta) + \eta^2\bar a_2Y_t^\eta.
			\label{eq:Y_eta_expand}
		\end{align}
	Since \eqref{eq:uniform_subcritical_eta}, we have
		$$
		2b+\rho-2\eta_0\bar a_1>0.
		$$
	Choose $\delta>0$ sufficiently small that
		\begin{equation}
			\label{eq:kappa_definition}
			\kappa := 2b+\rho - 2\eta_0\bar a_1 - 2\delta >0.
		\end{equation}
	Young's inequality yields
		\begin{equation}
			\label{eq:young_first_eta}
			2\eta\bar a_1 \mu(P_t^\eta)Y_t^\eta \leq \delta(Y_t^\eta)^2 + \frac{\eta^2\bar a_1^2}{\delta} \mu(P_t^\eta)^2.
		\end{equation}
	and
		\begin{equation}
			\label{eq:young_second_eta}
			\eta^2\bar a_2Y_t^\eta \leq \delta(Y_t^\eta)^2 + \frac{\eta^4\bar a_2^2}{4\delta},
		\end{equation}
	Since $\mu\geq0$,
		\begin{equation}
			\label{eq:young_third_eta}
			\eta^2\bar a_2\mu(P_t^\eta) \leq \frac{\eta^2\bar a_2}{2} \left( 1+\mu(P_t^\eta)^2 \right).
		\end{equation}
	Moreover,
		$$
		2 \eta \bar a_1(Y_t^\eta)^2 \leq 2\eta_0\bar a_1(Y_t^\eta)^2.
		$$
    Since $\eta\leq\eta_0$, $\eta^4 \leq \eta_0^2\eta^2$. Consequently, there exists a constant $A_0>0$, depending only on $\eta_0,\bar a_1,\bar a_2$ and $\delta$, but independent of $\eta$, $\alpha$, $t$, $T$, and $n$, such that
		\begin{align}
			&\bigl( \mu(P_t^\eta)+Y_t^\eta \bigr) \left( 2\eta\bar a_1Y_t^\eta + \eta^2\bar a_2 \right) \leq \bigl( 2\eta_0\bar a_1+2\delta \bigr) (Y_t^\eta)^2 + A_0\eta^2 \left( 1+\mu(P_t^\eta)^2 \right).
			\label{eq:Y_eta_integrand_bound}
		\end{align}
	Substituting \eqref{eq:Y_eta_integrand_bound} into \eqref{eq:Y_eta_squared_expectation_localized} and using \eqref{eq:kappa_definition}, we obtain
		\begin{align}
			& \mathbb E_{k,p,y_\eta}^{\alpha}
			\left[ e^{-\rho(T\wedge\tau_n)} \bigl(Y_{T\wedge\tau_n}^\eta\bigr)^2 \right]
			\notag\\
			&\quad+ \kappa \mathbb E_{k,p,y_\eta}^{\alpha} \left[ \int_0^{T\wedge\tau_n} e^{-\rho t}(Y_t^\eta)^2\,dt \right]
			\notag\\
			&\leq y_\eta^2 + A_0\eta^2 \mathbb E_{k,p,y_\eta}^{\alpha} \left[ \int_0^{T\wedge\tau_n} e^{-\rho t} \left( 1+\mu(P_t^\eta)^2 \right)dt \right].
			\label{eq:Y_eta_discounted_L2_localized}
		\end{align}
	Dropping the nonnegative terminal term and using the fact that $T\wedge\tau_n\leq T$ together with the nonnegativity of the integrand, we obtain
		\begin{align}
			\kappa \mathbb E_{k,p,y_\eta}^{\alpha}
			\left[ \int_0^{T\wedge\tau_n}
			e^{-\rho t}(Y_t^\eta)^2\,dt
			\right]
			&\leq y_\eta^2 + A_0\eta^2 \mathbb E_{k,p,y_\eta}^{\alpha} \left[ \int_0^T e^{-\rho t} \left( 1+\mu(P_t^\eta)^2 \right)dt \right].
			\label{eq:Y_eta_discounted_L2_localized_bound}
		\end{align}
	Using the assumptions \eqref{eq:y_eta_scaling} and \eqref{eq:uniform_discounted_mu_moment}, we obtain
		\begin{align}
			\kappa \mathbb E_{k,p,y_\eta}^{\alpha} \left[
			\int_0^{T\wedge\tau_n} e^{-\rho t}(Y_t^\eta)^2\,dt \right]
			&\leq \eta^2 \left[ c_y^2 + A_0 \left( \frac1\rho+B_\mu \right) \right].
			\label{eq:Y_eta_uniform_localized_bound}
		\end{align}
	Since $\tau_n\uparrow\infty$ a.s., for fixed $T$,
		$$
        \mathbf 1_{\{t\leq T\wedge\tau_n\}} \uparrow \mathbf 1_{\{t\leq T\}} \qquad\text{a.s.}
		$$
	Hence the monotone convergence theorem gives
		\begin{align}
			\kappa \mathbb E_{k,p,y_\eta}^{\alpha} \left[ \int_0^T e^{-\rho t}(Y_t^\eta)^2\,dt \right]
			&\leq \eta^2 \left[ c_y^2 + A_0 \left( \frac1\rho+B_\mu \right) \right].
			\label{eq:Y_eta_finite_T_L2}
		\end{align}
	Letting $T\to\infty$ and applying monotone convergence once more yields
		\begin{equation}
			\label{eq:Y_eta_discounted_L2}
			\mathbb E_{k,p,y_\eta}^{\alpha} \left[ \int_0^\infty e^{-\rho t}(Y_t^\eta)^2\,dt \right] \leq A_Y\eta^2,
		\end{equation}
		where
		\begin{equation}
			\label{eq:C_Y_definition}
			A_Y := \frac{ c_y^2+ A_0\left(\rho^{-1}+B_\mu\right)}{\kappa}.
		\end{equation}
	Since $A_Y$ is independent of the admissible control and of $\eta\in(0,\eta_0]$, we obtain
		\begin{equation}
			\label{eq:Y_eta_discounted_L2_uniform}
			\sup_{0<\eta\leq\eta_0} \sup_{\alpha\in\mathcal A_H(k,p,y_\eta)} \mathbb E_{k,p,y_\eta}^{\alpha} \left[ \int_0^\infty e^{-\rho t}(Y_t^\eta)^2\,dt \right] \leq A_Y\eta^2.
		\end{equation}
		
\item \textit{Weighted small-excitation estimate.}
	We define
		$$
		G_t^\eta := 1+ (K_t^{H,\eta})^q + (P_t^{H,\eta})^q.
		$$
	Since the integrand below is nonnegative, Tonelli's theorem and the Cauchy-Schwarz inequality yield
		\begin{align}
			&
			\mathbb E_{k,p,y_\eta}^{\alpha} \left[ \int_0^\infty e^{-\rho t}Y_t^\eta G_t^\eta\,dt \right] \notag\\
			&\leq \left( \mathbb E_{k,p,y_\eta}^{\alpha} \left[
			\int_0^\infty e^{-\rho t}(Y_t^\eta)^2\,dt \right]
			\right)^{1/2} \left( \mathbb E_{k,p,y_\eta}^{\alpha} \left[ \int_0^\infty e^{-\rho t}(G_t^\eta)^2\,dt \right] \right)^{1/2}.
			\label{eq:weighted_CS}
		\end{align}
	Applying the Cauchy-Schwarz inequality on the product space
	$$
	\bigl(\Omega\times\mathbb R_{\geq0}, \mathbb P_{k,p,y_\eta}^{\alpha}\otimes e^{-\rho t}dt \bigr),
	$$
	we obtain
	\begin{align}
		& \mathbb E_{k,p,y_\eta}^{\alpha}
		\left[ \int_0^\infty e^{-\rho t} Y_t^\eta G_t^\eta\,dt \right]
		\notag\\
		&\leq \left( \mathbb E_{k,p,y_\eta}^{\alpha} \left[
		\int_0^\infty e^{-\rho t} (Y_t^\eta)^2\,dt \right] \right)^{1/2} \left( \mathbb E_{k,p,y_\eta}^{\alpha} \left[ \int_0^\infty e^{-\rho t} (G_t^\eta)^2\,dt \right] \right)^{1/2}.
		\label{eq:weighted_CS}
	\end{align}
	Using $ (x_1+x_2+x_3)^2 \leq 3(x_1^2+x_2^2+x_3^2)$, we obtain
		$$ (G_t^\eta)^2 \leq 3 \left( 1+ (K_t^{H,\eta})^{2q} + (P_t^{H,\eta})^{2q} \right).
		$$
	Therefore, by
		\eqref{eq:uniform_discounted_state_moment},
		\begin{align}
			&\sup_{0<\eta\leq\eta_0} \sup_{\alpha\in\mathcal A_H(k,p,y_\eta)} \mathbb E_{k,p,y_\eta}^{\alpha} \left[\int_0^\infty e^{-\rho t} (G_t^\eta)^2\,dt \right] \leq 3B_q.
			\label{eq:G_eta_L2}
		\end{align}
	Combining \eqref{eq:Y_eta_discounted_L2}, \eqref{eq:weighted_CS}, and \eqref{eq:G_eta_L2}, we obtain
		\begin{align}
			& \mathbb E_{k,p,y_\eta}^{\alpha} \left[
			\int_0^\infty e^{-\rho t} Y_t^\eta \left( 1+ (K_t^{H,\eta})^q + (P_t^{H,\eta})^q \right)dt \right] \leq \sqrt{A_Y}\,\eta\, \sqrt{3B_q}.
		\end{align}
Therefore, setting $A_*:=\sqrt{3A_YB_q}$, and taking the supremum over $\alpha\in\mathcal A_H(k,p,y_\eta)$, yields
\begin{equation}
	\label{eq:weighted_excitation_eta_bound}
	\sup_{\alpha\in\mathcal A_H(k,p,y_\eta)} \mathbb E_{k,p,y_\eta}^{\alpha} \left[ \int_0^\infty e^{-\rho t} Y_t^\eta \left( 1 + (K_t^{H,\eta})^q + (P_t^{H,\eta})^q\right)dt \right] \leq A_*\eta.
\end{equation}
\end{enumerate}
\end{proof}

The preceding proposition controls the excitation component $Y^\eta$. To use this estimate in the comparison of the Hawkes and Poisson generators, we also need to control the change in the Poisson value function caused by a disaster jump. We recall the explicit Poisson value function obtained in \cite{sakhanda2025infinitehorizon}. Its jump increment admits a particularly simple growth estimate.

\begin{lemma}[Growth of the Poisson jump increment]
\label{lem:poisson_jump_increment_growth}
	
Assume that $0<\varepsilon<1$ and that
$$
0<\omega(k,p,\zeta)\leq1 \qquad \text{for all }(k,p,\zeta) \in \mathcal S_P\times\mathcal Z.
$$
Let
$$
v(k,p) = \frac{\psi^{-\varepsilon}}{1-\varepsilon} k^{1-\varepsilon} - \frac{x}{1+\beta}p^{1+\beta}, \qquad (k,p)\in\mathcal S_P,
$$
and define
$$
\Delta_\zeta v(k,p) := v\bigl(\omega(k,p,\zeta)k,p\bigr)-v(k,p).
$$
Then, for every $(k,p)\in\mathcal S_P$,
	\begin{equation}
		\label{eq:explicit_jump_increment_bound}
		\int_{\mathcal Z} \left| \Delta_\zeta v(k,p) \right| \nu(d\zeta) \leq \frac{\psi^{-\varepsilon}}{1-\varepsilon} k^{1-\varepsilon}.
	\end{equation}
Consequently,
	\begin{equation}
		\label{eq:weighted_jump_bound}
		\int_{\mathcal Z} \left| \Delta_\zeta v(k,p) \right| \nu(d\zeta) \leq A_v \left( 1+k^q+p^q \right),
		\qquad (k,p)\in\mathcal S_P,
	\end{equation}
with $q=1-\varepsilon $ and $A_v = \frac{\psi^{-\varepsilon}}{1-\varepsilon}$.
\end{lemma}

\begin{proof}
Fix $ (k,p)\in\mathcal S_P$ and $\zeta\in\mathcal Z$.
By the definition of $v$,
	\begin{align}
		\Delta_\zeta v(k,p)
		&= v\bigl(\omega(k,p,\zeta)k,p\bigr)-v(k,p)\notag\\
		&= \frac{\psi^{-\varepsilon}}{1-\varepsilon}
		\bigl(\omega(k,p,\zeta)k\bigr)^{1-\varepsilon} - \frac{x}{1+\beta}p^{1+\beta} - \frac{\psi^{-\varepsilon}}{1-\varepsilon} k^{1-\varepsilon} + \frac{x}{1+\beta}p^{1+\beta}.
	\end{align}
	The pollution terms cancel, and therefore
	\begin{equation}
		\label{eq:jump_increment_explicit}
		\Delta_\zeta v(k,p) = \frac{\psi^{-\varepsilon}}{1-\varepsilon}
		k^{1-\varepsilon} \left( \omega(k,p,\zeta)^{1-\varepsilon}-1 \right).
	\end{equation}
Since $0<\varepsilon<1$ and $0<\omega(k,p,\zeta)\leq1$,
$$
 -1 \leq \omega(k,p,\zeta)^{1-\varepsilon}-1 \leq0,
$$
hence
$$
\left| \omega(k,p,\zeta)^{1-\varepsilon}-1 \right|= 1-\omega(k,p,\zeta)^{1-\varepsilon} \leq1.
$$
Using \eqref{eq:jump_increment_explicit}, we obtain
$$
\left| \Delta_\zeta v(k,p) \right| \leq \frac{\psi^{-\varepsilon}}{1-\varepsilon} k^{1-\varepsilon}.
$$
The right-hand side is independent of $\zeta$. Since $\nu$ is a probability measure, $\nu(\mathcal Z)=1$, and therefore
	\begin{align}
		\int_{\mathcal Z} \left| \Delta_\zeta v(k,p) \right| \nu(d\zeta) \leq \frac{\psi^{-\varepsilon}}{1-\varepsilon} k^{1-\varepsilon} \int_{\mathcal Z}\nu(d\zeta) = \frac{\psi^{-\varepsilon}}{1-\varepsilon} k^{1-\varepsilon}.
	\end{align}
This proves \eqref{eq:explicit_jump_increment_bound}.
	
Finally, set $ q:=1-\varepsilon$. Since
$
k^q \leq 1+k^q+p^q,
$
we obtain
$$
\int_{\mathcal Z} \left| \Delta_\zeta v(k,p) \right| \nu(d\zeta) \leq \frac{\psi^{-\varepsilon}}{1-\varepsilon} \left( 1+k^q+p^q \right).
$$
Thus \eqref{eq:weighted_jump_bound} holds with $A_v = \frac{\psi^{-\varepsilon}}{1-\varepsilon}$.
\end{proof}

We now have the two estimates required for the comparison of the control problems. Proposition~\ref{prop:small_excitation_weighted_estimate}
shows that the discounted contribution of the Hawkes excitation is of order $\eta$, uniformly over admissible controls, while Lemma~\ref{lem:poisson_jump_increment_growth} provides the corresponding growth bound for the jump increment of the Poisson value function.

The key observation is that, when the Poisson value function is extended trivially to the enlarged Hawkes state space, the difference between the Hawkes and Poisson generators is entirely generated by the additional intensity $Y^\eta$. Consequently, the preceding two estimates allow this generator discrepancy to be controlled by a term of order $\eta$. We can therefore apply the Poisson verification function to the Hawkes dynamics and obtain matching upper and lower bounds for the Hawkes value function. This yields the main approximation result.

\begin{theorem}[Approximation of the Hawkes value function by the Poisson value function]
\label{thm:hawkes_poisson_value_approximation}

For every $\eta\in(0,\eta_0]$, let
	$$
	X^{H,\eta}=(K^{H,\eta},P^{H,\eta},Y^\eta)
	$$
be the Hawkes state process corresponding to the scaled excitation $a_\eta(\zeta):=\eta a(\zeta)$.

The excitation component satisfies
	$$
	dY_t^\eta=-bY_t^\eta\,dt+\int_{\mathcal Z}\eta a(\zeta)\,N^{H,\eta}(dt,d\zeta), \qquad Y_0^\eta=y_\eta,
	$$
where $N^{H,\eta}$ has predictable compensator $\bigl(\mu(P_{t-}^{H,\eta})+Y_{t-}^\eta\bigr)\nu(d\zeta)\,dt$.
	
For $\alpha\in\mathcal A_H(k,p,y_\eta)$, define
	$$
	J_H^\eta(k,p,y_\eta;\alpha) := \mathbb E_{k,p,y_\eta}^{\alpha} \left[ \int_0^\infty e^{-\rho t} U(K_t^{H,\eta},P_t^{H,\eta},\alpha_t)\,dt \right],
	$$
	and
	$$
	V_H^\eta(k,p,y_\eta) := \sup_{\alpha\in\mathcal A_H(k,p,y_\eta)} J_H^\eta(k,p,y_\eta;\alpha).
	$$
Let $v:\mathcal S_P\longrightarrow\mathbb R$ be the verified value function of the corresponding Poisson control problem, so that $v=V_P$.
Assume that $0<\varepsilon<1$ and
	$$
	v(k,p) = \frac{\psi^{-\varepsilon}}{1-\varepsilon}k^{1-\varepsilon} -
	\frac{x}{1+\beta}p^{1+\beta}.
	$$
Suppose that $v\in C^2(\mathcal S_P)$ and satisfies the Poisson HJB equation
	\begin{equation}
		\label{eq:Poisson_HJB_comparison}
		\rho v(k,p) = \sup_{\alpha\in\mathfrak a}
		\left\{ U(k,p,\alpha) + \mathcal L_P^\alpha v(k,p) \right\},
	\end{equation}
	where
	\begin{align}
		\mathcal L_P^\alpha v(k,p)
		&= b^{\mathrm{cap}}(k,p,\alpha)v_k(k,p)	+
		b^{\mathrm{pol}}(k,p,\alpha)v_p(k,p) + \frac12\sigma^2p^2v_{pp}(k,p) \notag\\
		&\quad+ \mu(p)\int_{\mathcal Z}\Delta_\zeta v(k,p)\,\nu(d\zeta),
	\end{align}
	with
	$$
	\Delta_\zeta v(k,p) := v\bigl(\omega(k,p,\zeta)k,p\bigr)-v(k,p).
	$$
	Assume the hypotheses of Proposition~\ref{prop:small_excitation_weighted_estimate}, with $q=1-\varepsilon$, and the hypotheses of Lemma~\ref{lem:poisson_jump_increment_growth}.
	
	Assume furthermore that, for every $\eta\in(0,\eta_0]$ and every
	$\alpha\in\mathcal A_H(k,p,y_\eta)$,
	\begin{enumerate}
		\item
		\begin{equation}
			\label{eq:utility_integrability_H_eta}
			\mathbb E_{k,p,y_\eta}^{\alpha} \left[ \int_0^\infty e^{-\rho t} \left| U(K_t^{H,\eta},P_t^{H,\eta},\alpha_t) \right|dt \right] <\infty.
		\end{equation}
		
		\item
		The transversality condition holds:
		\begin{equation}
			\label{eq:hawkes_transversality_eta}
			\lim_{T\to\infty} \mathbb E_{k,p,y_\eta}^{\alpha} \left[ e^{-\rho T} \left| v(K_T^{H,\eta},P_T^{H,\eta}) \right| \right] =0
		\end{equation}
		
		\item
		For every $T<\infty$, the family
		\begin{equation}
			\label{eq:UI_stopped_v_eta}
			\left\{ e^{-\rho\tau} v(K_\tau^{H,\eta},P_\tau^{H,\eta}) : \tau\leq T \text{ is a stopping time} \right\}
		\end{equation}
	is uniformly integrable.
	\end{enumerate}
Finally, assume that there exists a Borel measurable selector $\widehat\alpha_P:\mathcal S_P\longrightarrow\mathfrak a$
attaining the supremum in \eqref{eq:Poisson_HJB_comparison}, and that, for every $\eta\in(0,\eta_0]$, the feedback control
	$$
	\widehat\alpha_{P,t}^{\,\eta} := \widehat\alpha_P(K_t^{H,\eta},P_t^{H,\eta})
	$$
belongs to $\mathcal A_H(k,p,y_\eta)$. Then, for every $\eta\in(0,\eta_0]$,
	\begin{equation}
		\label{eq:value_approximation_eta} \left| V_H^\eta(k,p,y_\eta)-V_P(k,p) \right| \leq \frac{\psi^{-\varepsilon}}{1-\varepsilon}A_*\eta,
	\end{equation}
where $A_*$ is the constant appearing in Proposition~\ref{prop:small_excitation_weighted_estimate}.
Consequently,
	$$
	V_H^\eta(k,p,y_\eta) \longrightarrow V_P(k,p)
	\qquad\text{as } \eta\downarrow 0.
	$$
\end{theorem}

\begin{proof}
Set $ A_v := \frac{\psi^{-\varepsilon}}{1-\varepsilon}$ and $q:=1-\varepsilon$. By Lemma~\ref{lem:poisson_jump_increment_growth},
	\begin{equation}
		\label{eq:weighted_jump_bound_from_lemma}
		\int_{\mathcal Z} \left| \Delta_\zeta v(k,p) \right| \nu(d\zeta) \leq A_v \left( 1+k^q+p^q \right), \qquad (k,p)\in\mathcal S_P.
	\end{equation}
Moreover, Proposition~\ref{prop:small_excitation_weighted_estimate} gives
	\begin{equation}
		\label{eq:weighted_excitation_from_prop}
		\sup_{\alpha\in\mathcal A_H(k,p,y_\eta)} \mathbb E_{k,p,y_\eta}^{\alpha} \left[ \int_0^\infty e^{-\rho t} Y_t^\eta \left( 1+ (K_t^{H,\eta})^q+ (P_t^{H,\eta})^q \right)dt \right] \leq
		A_*\eta.
	\end{equation}
We divide the remainder of the proof into five steps.
\begin{enumerate}   
	\item \textit{Extension of the Poisson value function to the Hawkes state space.} We define $ \widetilde v:\mathcal S_H \longrightarrow \mathbb R$ by
	$$
	\widetilde v(k,p,y):=v(k,p).
	$$
	In particular, $ \partial_y\widetilde v(k,p,y)=0$. For a sufficiently regular test function $\varphi:\mathcal S_H\to\mathbb R$, the controlled generator of the Hawkes state process is
	\begin{align}
		\mathcal L_{H, \eta}^\alpha\varphi(k,p,y)
		=& b^{\mathrm{cap}}(k,p,\alpha) \partial_k\varphi(k,p,y) + b^{\mathrm{pol}}(k,p,\alpha) \partial_p\varphi(k,p,y) \notag\\
		&+ \frac12\sigma^2p^2 \partial_{pp}\varphi(k,p,y) - by\,\partial_y\varphi(k,p,y) \notag\\
		&+\bigl(\mu(p)+y\bigr) \int_{\mathcal Z} \left[ \varphi \bigl( \omega(k,p,\zeta)k, p, y+\eta a(\zeta) \bigr) - \varphi(k,p,y) \right] \nu(d\zeta). 
		\label{eq:Hawkes_full_generator}
	\end{align}
This is the standard generator associated with a jump-diffusion having state-dependent marked jump intensity; see, for example, \cite{applebaum_levy_2009,oksendal_applied_2019}.
Since $\widetilde v$ is independent of $y$,
$$
\widetilde v\bigl(\omega(k,p,\zeta)k,p,y+\eta a(\zeta)\bigr) - \widetilde v(k,p,y) = \Delta_\zeta v(k,p).
$$
Therefore, recalling the definition of the Poisson generator
$\mathcal L_P^\alpha$,
\begin{equation}
\mathcal L_{H,\eta}^\alpha\widetilde v(k,p,y) = \mathcal L_P^\alpha v(k,p) + y\int_{\mathcal Z} \Delta_\zeta v(k,p)\,\nu(d\zeta).
\label{eq:exact_generator_difference}
\end{equation}
Hence, by Lemma~\ref{lem:poisson_jump_increment_growth},
\begin{equation}
\left|\mathcal L_{H,\eta}^\alpha\widetilde v(k,p,y) - \mathcal L_P^\alpha v(k,p) \right| \leq A_vy\bigl(1+k^q+p^q\bigr).
\label{eq:generator_weighted_bound}
\end{equation}

	\item \textit{Residual of the Poisson HJB under the Hawkes generator.} Since $v$ satisfies \eqref{eq:Poisson_HJB_comparison}, for every admissible $\alpha\in\mathfrak a$,
	\begin{equation}
		U(k,p,\alpha) + \mathcal L_P^\alpha v(k,p) - \rho v(k,p) \leq 0.
		\label{eq:Poisson_HJB_inequality}
	\end{equation}
	Combining this inequality with \eqref{eq:exact_generator_difference}, we obtain
	\begin{align*}
		&U(k,p,\alpha) + \mathcal L_{H, \eta}^\alpha\widetilde v(k,p,y) - \rho\widetilde v(k,p,y) \\
		&= U(k,p,\alpha) + \mathcal L_P^\alpha v(k,p) - \rho v(k,p) + y\int_{\mathcal Z} \Delta_\zeta v(k,p)\,\nu(d\zeta) \\
		&\leq y\int_{\mathcal Z} \Delta_\zeta v(k,p)\,\nu(d\zeta).
	\end{align*}
	Therefore, using \eqref{eq:weighted_jump_bound},
	\begin{equation}
		U(k,p,\alpha) + \mathcal L_{H, \eta}^\alpha\widetilde v(k,p,y) - \rho\widetilde v(k,p,y) \leq A_vy\bigl(1+k^q+p^q\bigr).
		\label{eq:Hawkes_residual_bound}
	\end{equation}

	\item \textit{Localization and It\^o's formula.} Fix an arbitrary control $ \alpha\in\mathcal A_H(k,p,y)$. Let
	$$
	\widetilde N^{H,\eta}(dt,d\zeta) := N^{H,\eta}(dt,d\zeta) - \bigl( \mu(P_{t-}^{H,\eta})+Y_{t-}^\eta \bigr) \nu(d\zeta)\,dt
	$$
	denote the compensated marked random measure associated with $N^{H, \eta}$. Applying the It\^o formula for semimartingales with jumps to $e^{-\rho t} v(K_t^{H,\eta},P_t^{H,\eta})$ the local martingale terms are
	\begin{align}
		M_t^W &= \int_0^t e^{-\rho s} \sigma P_s^{H,\eta} v_p(K_s^{H,\eta},P_s^{H,\eta})\,dW_s, \\
		M_t^{N,\eta} &= \int_{(0,t]\times\mathcal Z} e^{-\rho s} \Delta_\zeta v(K_{s-}^{H,\eta},P_{s-}^{H,\eta}) \widetilde N^{H,\eta}(ds,d\zeta).
		\label{eq:jump_martingale_comparison}
	\end{align}
	The It\^o formula for semimartingales with jumps and the associated stochastic integration theory are standard; see \cite{protter_stochastic_2005}. For stochastic integration with Poisson random measures and L\'evy-driven SDEs, see also \cite{applebaum_levy_2009}; for the corresponding jump-diffusion control framework, see \cite{oksendal_applied_2019}. 
	
	Since $M^W$ and $M^{N,\eta}$ are local martingales, there exists an increasing sequence of stopping times $(\tau_n)_{n\geq1}$ such that
	$$
	\tau_n\uparrow\infty
	\qquad\text{a.s.}
	$$
	such the stopped processes $M^W_{\cdot\wedge\tau_n}$, $M^{N, \eta}_{\cdot\wedge\tau_n}$ are true martingales; see, for example, \cite{protter_stochastic_2005}.

Therefore, for every $T<\infty$, taking expectations in It\^o's formula gives
		\begin{align}
		&\mathbb E_{k,p,y_\eta}^{\alpha} \left[ e^{-\rho(T\wedge\tau_n)} v(K_{T\wedge\tau_n}^{H,\eta}, P_{T\wedge\tau_n}^{H,\eta}) \right] - v(k,p)
		\notag\\
		&= \mathbb E_{k,p,y_\eta}^{\alpha} \left[ \int_0^{T\wedge\tau_n} e^{-\rho t} \left[ \mathcal L_{H,\eta}^{\alpha_t}\widetilde v (K_{t-}^{H,\eta},P_{t-}^{H,\eta},Y_{t-}^\eta) - \rho v(K_{t-}^{H,\eta},P_{t-}^{H,\eta}) \right]dt \right].
		\label{eq:localized_expectation_eta}
	\end{align}
Using \eqref{eq:Hawkes_residual_bound}, we obtain
	\begin{align}
		&\mathbb E_{k,p,y_\eta}^{\alpha}
		\left[ \int_0^{T\wedge\tau_n} e^{-\rho t} U(K_t^{H,\eta},P_t^{H,\eta},\alpha_t)\,dt \right]
		\notag\\
		&\leq v(k,p) - \mathbb E_{k,p,y_\eta}^{\alpha} \left[ e^{-\rho(T\wedge\tau_n)} v(K_{T\wedge\tau_n}^{H,\eta}, P_{T\wedge\tau_n}^{H,\eta}) \right]
		\notag\\
		&\quad+ A_v \mathbb E_{k,p,y_\eta}^{\alpha} \left[
		\int_0^{T\wedge\tau_n} e^{-\rho t} Y_t^\eta \left( 1+ (K_t^{H,\eta})^q+ (P_t^{H,\eta})^q \right)dt \right].
		\label{eq:localized_upper_eta}
	\end{align}
    
    \item \textit{Removal of localization and passage to the infinite horizon.} Fix $T<\infty$. Since $ \tau_n \uparrow \infty$ a.s., we have
	$$
	T\wedge\tau_n\longrightarrow T \qquad\text{a.s.}  \text{ as } n\to\infty .
	$$
	By \eqref{eq:utility_integrability_H_eta}, dominated convergence gives
	$$
	\mathbb E^\alpha \left[ \int_0^{T\wedge\tau_n} e^{-\rho t}
	U(K_t^{H,\eta},P_t^{H,\eta},\alpha_t)\,dt \right] \longrightarrow \mathbb E^\alpha \left[ \int_0^T e^{-\rho t} U(K_t^{H,\eta},P_t^{H,\eta},\alpha_t)\,dt \right].
	$$
	
	Since the weighted excitation integrand is nonnegative, monotone
	convergence gives
	\begin{align}
		& \mathbb E^\alpha
		\left[ \int_0^{T\wedge\tau_n} e^{-\rho t} Y_t^\eta \left(1+(K_t^{H,\eta})^q+ (P_t^{H,\eta})^q \right)dt \right]
		\notag\\
		&\longrightarrow \mathbb E^\alpha \left[ \int_0^T e^{-\rho t} Y_t^\eta \left( 1+ (K_t^{H,\eta})^q+ (P_t^{H,\eta})^q \right)dt \right].
	\end{align}
	
	Moreover,
	$$
	e^{-\rho(T\wedge\tau_n)} v(K_{T\wedge\tau_n}^{H,\eta}, P_{T\wedge\tau_n}^{H,\eta}) \longrightarrow e^{-\rho T} v(K_T^{H,\eta},P_T^{H,\eta}) \qquad\text{a.s.},
    $$
	and the uniform-integrability assumption \eqref{eq:UI_stopped_v_eta} implies convergence in $L^1$. Thus, letting $n\to\infty$ in \eqref{eq:localized_upper_eta},
	\begin{align}
		& \mathbb E^\alpha \left[ \int_0^T e^{-\rho t} U(K_t^{H,\eta},P_t^{H,\eta},\alpha_t)\,dt \right] \notag\\
		&\leq v(k,p) - \mathbb E^\alpha \left[ e^{-\rho T} v(K_T^{H,\eta},P_T^{H,\eta}) \right] \notag\\
		&\quad+ A_v \mathbb E^\alpha \left[ \int_0^T e^{-\rho t}
		Y_t^\eta \left( 1+ (K_t^{H,\eta})^q+ (P_t^{H,\eta})^q \right)dt \right].
		\label{eq:finite_T_upper_eta}
	\end{align}

    Letting $T\to\infty$, the utility term converges by \eqref{eq:utility_integrability_H_eta}, the terminal term vanishes by \eqref{eq:hawkes_transversality_eta}, and the nonnegative weighted excitation term converges by monotone convergence. Hence
	\begin{align}
		J_H^\eta(k,p,y_\eta;\alpha)
		&\leq v(k,p) + A_v \mathbb E_{k,p,y_\eta}^{\alpha} \left[ \int_0^\infty e^{-\rho t} Y_t^\eta \left( 1+ (K_t^{H,\eta})^q+ (P_t^{H,\eta})^q \right)dt \right].
		\label{eq:fixed_control_upper_eta}
	\end{align}
	
	Since $\alpha$ was arbitrary, Proposition~\ref{prop:small_excitation_weighted_estimate} gives
	\begin{equation}
		\label{eq:VH_upper_eta} 
		V_H^\eta(k,p,y_\eta) \leq v(k,p)+A_vA_*\eta.
	\end{equation}

	\item \textit{Lower estimate using the Poisson-optimal feedback.}
	Let $\widehat\alpha_P:\mathcal S_P\longrightarrow\mathfrak a$ be the measurable feedback attaining the supremum in the Poisson HJB. Thus,
	\begin{equation}
		U(k,p,\widehat\alpha_P(k,p)) + \mathcal L_P^{\widehat\alpha_P(k,p)}v(k,p) - \rho v(k,p) = 0.
		\label{eq:HJB_equality_optimal_control}
	\end{equation}
	The use of measurable maximizing feedbacks and the corresponding verification argument is standard in stochastic control; see \cite{fleming_controlled_2006, pham_continuous-time_2009}, and, in the jump-diffusion setting, \cite{oksendal_applied_2019}.
	By assumption, the feedback control $ \widehat\alpha_{P,t}^{\,\eta} = \widehat\alpha_P (K_t^{H,\eta},P_t^{H,\eta})$ is admissible for the Hawkes problem. Combining \eqref{eq:HJB_equality_optimal_control} and \eqref{eq:exact_generator_difference}, we obtain
	\begin{align}
		U(k,p,\widehat\alpha_P(k,p)) + \mathcal L_{H, \eta}^{\widehat\alpha_P(k,p)} \widetilde v(k,p,y) - \rho\widetilde v(k,p,y) = y \int_{\mathcal Z} \Delta_\zeta v(k,p)\,\nu(d\zeta).
		\label{eq:exact_residual_optimal}
	\end{align}
Repeating the localization and limiting arguments above under the
control $\widehat\alpha_P^{\,\eta}$ gives
\begin{align}
    &J_H^\eta(k,p,y_\eta;\widehat\alpha_P^{\,\eta}) - v(k,p) = \mathbb E_{k,p,y_\eta}^{\widehat\alpha_P^{\,\eta}} \left[ \int_0^\infty e^{-\rho t}Y_t^\eta \int_{\mathcal Z} \Delta_\zeta v (K_t^{H,\eta},P_t^{H,\eta}) \,\nu(d\zeta)\,dt \right].
    \label{eq:exact_error_identity_eta}
\end{align}
By the triangle inequality, Lemma~\ref{lem:poisson_jump_increment_growth}, and
Proposition~\ref{prop:small_excitation_weighted_estimate},
\begin{align}
    &\left| J_H^\eta (k,p,y_\eta;\widehat\alpha_P^{\,\eta}) - v(k,p) \right| \notag\\
    &\qquad\leq A_v \mathbb E_{k,p,y_\eta}^{\widehat\alpha_P^{\,\eta}} \left[ \int_0^\infty e^{-\rho t}Y_t^\eta \left( 1+(K_t^{H,\eta})^q+(P_t^{H,\eta})^q \right)dt \right] \notag\\
    &\qquad\leq A_vA_*\eta.
\end{align}
Hence,
$$
J_H^\eta (k,p,y_\eta;\widehat\alpha_P^{\,\eta}) \geq v(k,p)-A_vA_*\eta.
$$
Since $\widehat\alpha_P^{\,\eta}$ is admissible for the Hawkes problem,
$$
V_H^\eta(k,p,y_\eta) \geq  J_H^\eta (k,p,y_\eta;\widehat\alpha_P^{\,\eta}) \geq v(k,p)-A_vA_*\eta.
$$
Together with \eqref{eq:VH_upper_eta}, this yields
$$
v(k,p)-A_vA_*\eta \leq V_H^\eta(k,p,y_\eta) \leq v(k,p)+A_vA_*\eta.
$$
Therefore,
\begin{equation}
    \left| V_H^\eta(k,p,y_\eta)-v(k,p) \right| \leq A_vA_*\eta.
    \label{eq:value_error_eta_final}
\end{equation}
Finally, since $v=V_P$ and
$
A_v=\frac{\psi^{-\varepsilon}}{1-\varepsilon},
$
we conclude that
$$
\left| V_H^\eta(k,p,y_\eta)-V_P(k,p) \right| \leq \frac{\psi^{-\varepsilon}}{1-\varepsilon}A_*\eta.
$$
Letting $\eta\downarrow0$ yields
$$
V_H^\eta(k,p,y_\eta) \longrightarrow V_P(k,p).
$$
\end{enumerate}
\end{proof}

Theorem~\ref{thm:hawkes_poisson_value_approximation} therefore provides a quantitative justification for using the Poisson control problem as an approximation of the Hawkes problem in the small-excitation regime. The uniform weighted estimate over admissible controls leads to an $O(\eta)$ bound for the difference between the Hawkes and Poisson value functions. In particular, the Poisson value function arises as the limiting value function as the magnitude of Hawkes self-excitation tends to zero.

\bibliographystyle{plain}
\bibliography{references}

\end{document}